\documentclass[11pt,letterpaper]{amsart}

\usepackage[english]{babel}
\usepackage{amsmath}
\usepackage{amssymb}
\usepackage{graphicx}
\usepackage{hyperref,latexsym}
\usepackage[dvipsnames]{xcolor}
\usepackage{pdfpages}
\usepackage{subcaption}
\usepackage{mathtools}

\newcommand{\backassign}{=:}
\newcommand{\dueto}[1]{\textup{\textbf{(#1) }}}
\newcommand{\infixand}{\text{ and }}
\newcommand{\longdownarrow}{{\mbox{\rotatebox[origin=c]{-90}{$\longrightarrow$}}}}
\newcommand{\longuparrow}{{\mbox{\rotatebox[origin=c]{90}{$\longrightarrow$}}}}
\newcommand{\mathD}{\mathrm{D}}
\newcommand{\mathd}{\mathrm{d}}
\newcommand{\mathi}{\mathrm{i}}
\newcommand{\nin}{\not\in}
\newcommand{\nobracket}{}
\newcommand{\tmmathbf}[1]{\ensuremath{\boldsymbol{#1}}}
\newcommand{\tmname}[1]{\textsc{#1}}

\newcommand{\tmop}[1]{\ensuremath{\operatorname{#1}}}
\newcommand{\tmtextit}[1]{\text{{\itshape{#1}}}}
\newcommand{\assign}{:=}

\newcommand{\tmdummy}{$\mbox{}$}

\newenvironment{proof*}[1]{\noindent\textbf{#1\ }}{\hspace*{\fill}$\Box$\medskip}
\newtheorem{theorem}{Theorem}[section]
\newtheorem{lemma}[theorem]{Lemma}
\newtheorem{remark}[theorem]{Remark}
\newtheorem{definition}[theorem]{Definition}
\newtheorem*{definition*}{Definition}
\newtheorem{conjecture}[theorem]{Conjecture}

\newcommand{\Sha}{\mathcal{S}_0}

\numberwithin{equation}{section}

\makeatletter
\def\paragraph{\@startsection{paragraph}{4}%
  \z@\z@{-\fontdimen2\font}%
  {\normalfont\bfseries}}
\makeatother

\begin{document}

\title[Dyadic reduction of the UMD conjecture]{A two sided linear estimate and a dyadic reduction of the UMD conjecture}

\author{Komla Domelevo}
\address{Institute of Mathematics, University of Würzburg, Germany}

\author{Stefanie Petermichl}
\thanks{S.P. is partially supported by the Alexander von Humboldt foundation}
\address{Institute of Mathematics, University of Würzburg, Germany}

\begin{abstract}
  We define a time faithful dyadic shift operator of complexity one, that is an antisymmetric antiinvolution.  We show that the Hilbert transform with values in a Banach space is $L^p$ bounded if and only if the dyadic shift is -- with a linear two sided norm dependence. The results reduce the famous UMD conjecture to a pair of simple dyadic operators.
\end{abstract}

\maketitle

\section{Introduction} \label{S: introduction}
The study of vector-valued singular integral operators was begun in the 1980’s, when {\tmname{Burkholder}} {\cite{Bur1983a}} and {\tmname{Bourgain}}
{\cite{Bou1983a}}  characterized the equivalence of a geometric property of a Banach space X and the boundedness of the tensor extension of the Hilbert transform $\mathcal{H}$ to X-valued $L^p$-functions. This  property of the Banach space X is now known as the UMD property and can be defined via unconditional convergence of martingale differences.

 The UMD spaces are therefore the right setting beyond Hilbert spaces, where one can sensibly extend classical results from the Calder\'{o}n-Zygmund and Littlewood-Paley theories to vector-valued functions. The UMD property also became important in connection to partial differential equations or numerical analysis, where functions take values in Banach spaces. UMD spaces have been intensively studied, we refer here to the text books by {\tmname{Pisier}} \cite{Pis2016}, {\tmname{Hyt\"{o}nen}}--{\tmname{van Neerven}}--{\tmname{Veraar}}--{\tmname{Weis}} \cite{HytNeeVerWei2016a} for an excellent introduction to the basic properties and beyond.
 
 \

It is thus a known fact that the Hilbert transform $\mathcal{H}$ acting on functions with values in a Banach space $X$ is
bounded if and only if $X$ has the UMD property. One way to define the latter is by using the sign toss operators acting on the Haar system $\{h_I:I\in \mathcal{D}\}$ by 
\[ \mathcal{T}_{\alpha} : h_{I} \mapsto \alpha_I h_{I} , \]
where $\alpha_I \in \{-1,+1\}$: $X$ has the UMD--$p$ property if and only if $\mathcal{T}_{\alpha}$ are uniformly $L^p$ bounded. The best uniform bound is noted the UMD--$p$ constant of $X$.

The precise relationship between the norm of the Hilbert transform and the UMD constant, however, remains unclear. If the $L^p$ norm of the Hilbert transform is noted $h_p$ and the UMD--$p$
constant is noted $m_p$, then it is known that
\begin{equation}
\label{eq: quadratic bounds}
m^{1/2}_p \leqslant h_p \leqslant m^2_p.
\end{equation}
The estimate on the left hand side is due to {\tmname{Bourgain}}
{\cite{Bou1983a}} and the estimate on the right hand side is due to
{\tmname{Burkholder}} {\cite{Bur1983a}}. It is an open question whether these
inequalities can be improved, ideally if they are linear. This question is known as the famous UMD conjecture, see for example O.6 in \cite{HytNeeVerWei2016a}: Is it true that 
\begin{equation}\label{eq: UMD conjecture}
m_p \lesssim h_p \lesssim m_p?
\end{equation}

\
In the paper \cite{GeiMonSak2010}, {\tmname{Geiss}}--{\tmname{Montgomery-Smith}}--{\tmname{Saksman}} replaced the Hilbert transform by the even singular operator $\mathcal{R}^2_1-\mathcal{R}^2_2$, the difference of squares of Riesz transforms in $\mathbb{R}^2$, and showed that the linear relation holds. If $r_p$ denotes the $L^p$ norm of $\mathcal{R}^2_1-\mathcal{R}^2_2$ with values in $X$, they showed 
\begin{equation}
\label{eq: linear bounds even}
 m_p \le r_p \le m_p.
 \end{equation}
Their remarkable estimates rely heavily on evenness of the operator $\mathcal{R}^2_1-\mathcal{R}^2_2$ as well as other previously observed mapping properties and relations with dyadic martingales. Other than the Hilbert transform, the difference of squares of Riesz transforms are the expectation of a martingale multiplier using a diagonal matrix, which facilitates working with the complexity 0 operators $\mathcal{T}_{\alpha}$.

\

In this paper we work with the Hilbert transform $\mathcal{H}$ but replace $\mathcal{T}_{\alpha}$ by the odd dyadic shift operator of complexity 1 densely defined by
\begin{equation}
\label{eq: definition of Sha}
 \Sha : h_{I_{\pm}} \mapsto \pm h_{I_{\mp}},
\end{equation}
where $I_\pm$ denote the left and right children of $I$ and $h_{I_\pm}$ their associated Haar functions.
If $\Sha$ has the $L^p$ bound $s_p$ and $\mathcal{H}$ has the $L^p$ bound $h_p$, our main results are the two sided linear bounds, as opposed to the quadratic bounds in \eqref{eq: quadratic bounds}. We show that 
\begin{equation}
\label{eq: linear bounds}
 s_p \lesssim h_p \le s_p,
 \end{equation}
where we strongly stress the exponents $1$ on $s_p$ and remark the coefficient $1$ on the right hand side. The coefficient on the left hand side is explicit as well, but smaller than 1. The estimates \eqref{eq: linear bounds} completely reduce the UMD conjecture \eqref{eq: UMD conjecture} (its failure or its proof) to the pair of dyadic operators $\mathcal{T}_\alpha$ and $\Sha$. See the next section for the formal statement of the results.

\

When the Banach space is the Hilbert space $\mathbb{R}$, then the situation is much simpler and relations are linear. Indeed, the sharp norm estimates depending on $p$ are known for a long time \cite{Bur1984a},\cite{Pic1972},\cite{Ess1984}. Observe that the norms of $\mathcal{T}_{\alpha}$ and $\mathcal{H}$ are not the same, so the linear relations have coefficients different from 1. The even operator $\mathcal{R}^2_1-\mathcal{R}^2_2$ considered in \cite{GeiMonSak2010}, however, has exactly the same $L^p$ bound as $\mathcal{T}_{\alpha}$, further emphasizing the relative simplicity of the even case. Both numeric estimates (i.e. the coefficients) in our estimates \eqref{eq: linear bounds} are new, even if $X=\mathbb{R}$. Indeed, it is new to obtain any estimate proportional to $h_p$ for a dyadic operator as it has previously not been possible to use orthogonality in a dyadic setting when estimating any shift operator. See the next paragraphs for more details.

\

One should appreciate that the definition of a UMD space involves a form of differential subordination condition much more restrictive than its analog in Hilbert spaces, as it targets each direction separately as opposed to a norm estimate on the increments (see e.g. {\tmname{Burkholder}} \cite{Bur2001a} for a discussion). For this and other reasons, the even case is decisively easier than the odd case considered here: the Hilbert transform forces dyadic operators with memory - with such a weak subordination condition, increments from different paths or at different times cannot be compared or estimated against one another. It is the aspect of the memory as present in our dyadic processes that causes difficulty and forces novel ideas in both of our estimates. In \cite{GeiMonSak2010} the authors attempted to treat the odd case, but their considerations replaced one continuous process by another (hence no issues with memory or jumps as a dyadic reduction is absent). The authors showed that the Hilbert transform is bounded if and only if its stochastic representation is - we implicitly use this here without mention. Further, the Hilbert transform is tied to the principle of orthogonality and we are in a setting where the concept of angles is a priori absent. Yet in our novel strategies for upper and lower estimates, we develop a novel concept of dyadic orthogonality to obtain our estimates.

\

\paragraph{History of dyadic shifts}
The dyadic martingale transform $\mathcal{T}_{\alpha}$ has often served as a simpler analog of singular integral operators such as the Hilbert transform. Aside from 
{\cite{Bou1983a}} and {\cite{Bur1983a}} there are numerous other examples, {\cite{Nazarov}}, {\cite{TreVol1997}}, {\cite{PV2002}}, to name a few. But its precision is not always sufficient, especially if one seeks quantitative or optimal bounds. In such cases, the classical Haar shift $\mathcal{S}_{\tmop{cl}}$ from \cite{Pet2000} has had great success: 
\[\mathcal{S}_{\tmop{cl}} : h_I \rightarrow (h_{I_+} -
h_{I_-}) / \sqrt{2}.\]

In particular it is proven by {\tmname{Petermichl}} in \cite{Pet2000} that a non zero multiple of the Hilbert transform is an average of classical Haar shifts over different dyadic grids. Through averaging it was immediately clear that $h_p \lesssim s^{\tmop{cl}}_p$ if
$s^{\tmop{cl}}_p$ is the $L^p$ bound for the classical $X$--valued Haar shift. Notice that since the main tool here is taking an expectation, only the upper estimate is linear. It was also shown by
{\tmname{Petermichl}}--{\tmname{Pott}} {\cite{PetPot2003a}} that $s^{\tmop{cl}}_p
\leqslant m^2_p$. Together, a very short and simple deterministic alternative to {\tmname{Burkholder}}'s direction in \ref{eq: quadratic bounds} with the quadratic bound was thus obtained. But meanwhile the theory had been extended significantly: by means of a different type of discretization, {\tmname{Figiel}} proved in \cite{Fig1990} the same quadratic estimate for more general Calder\'{o}n--Zygmund operators. Through the use of the generalization of the Haar shift by {\tmname{Hyt\"{o}nen}} \cite{H} to represent Calder\'{o}n--Zygmund operators, {\tmname{Pott}}--{\tmname{Stoica}} were able to get linear upper bounds for a large class of {\it even} Calder\'{o}n--Zygmund operators in \cite{Pot2014}.

\

The classical
shift has proven a useful model for the Hilbert transform for some
important applications outside the geometry of Banach spaces. Among others, it solved two open questions, delivering at least two optimal bounds. One is found in {\tmname{Petermichl}} \cite{Pet2000} and {\tmname{Nazarov--Pisier--Treil--Volberg}} \cite{NPTV}, where the precise dimensional growth of a Hankel operator with matrix symbol by means of \tmname{Pisier}'s strong operator BMO norm as well as related estimates involving paraproduct operators with matrix symbol, a matrix Carleson embedding,  and  $H^1$ multipliers was established. The Haar shift was instrumental in closing the loop between these objects. Another application was in the theory of weights,  where the so-called $A_2$ conjecture was solved by {\tmname{Petermichl}} in \cite{Pet2007}, that is, the optimal bound for the Hilbert transform in weighted $L^2$ spaces in terms of the Muckenhoupt characteristic of the weight. The idea of the shift operators has been generalized in {\tmname{Petermichl--Treil--Volberg}} \cite{PTV} and {\tmname{Hyt\"{o}nen}} \cite{H}. The latter gives a beautiful representation for all Calder\'{o}n--Zygmund operators by developing deep ideas from {\tmname{Nazarov--Treil--Volberg}} \cite{NTVnonhom} and resolved the $A_2$ problem for all Calder\'{o}n--Zygmund operators. It is interesting to note, and somewhat mirrors what we see here in the subject of this paper, that the optimal weighted bound for the {\it even} Beurling--Ahlfors transform was proved before that for the Hilbert transform with a novel but simpler proof by {\tmname{Petermichl}}--{\tmname{Volberg}} \cite{PV2002}, via the {\it even} dyadic martingale multiplier. However, the numeric $L^p$ bound for the Beurling--Ahlfors transform, conjectured to be identical to the uniform $L^p$ bound of $\mathcal{T}_\alpha$, remains a puzzling open question.

\

\paragraph{The choice of the dyadic model.}
The classical shift lacks a number of defining properties and similarities with the Hilbert transform. 
The novel dyadic $\Sha$ defined in \eqref{eq: definition of Sha} does average to a $0$ multiple of the Hilbert transform via the ideas in \cite{Pet2000} and
is therefore not applicable in the same way as the classical shift.
But it is in other ways incomparably closer to the Hilbert transform. Its square is the negative identity, it is antisymmetric and -- apparently important for this subject -- it has no even component. Another decisive feature is that it is time faithful, meaning that the coefficient swap happens at the same interval size or time. Endowed with the right sign swap on this exchange, one obtains a form of dyadic orthogonality that is crucial to our arguments. The classical shift has none of these features.

It is interesting to recall that {\it even} $X$--valued singular operators have posed fewer problems in  {\tmname{Geiss}}--{\tmname{Montgomery-Smith}}--{\tmname{Saksman}} \cite{GeiMonSak2010} because the defining operator for the UMD property, $\mathcal{T}_\alpha$, has the natural modeling properties for $\mathcal{R}^2_1-\mathcal{R}^2_2$, as was beautifully observed by the authors in their text.

\paragraph{The lower bound.} The proof of Theorem \ref{theorem_lower} requires the perfect choice of dyadic model that responds to the needs of the Hilbert transform. Indeed, $\mathcal{S}_0$ mimics the
Cauchy--Riemann equations `in the probability space'. To obtain the lower bound, we add several new elements to a brilliant
argument by {\tmname{Bourgain}} \cite{Bou1983a} using high frequency modulation in Fourier space. In his work, he
needed to apply the Hilbert transform twice to control the martingale transforms. Our relationship of the Hilbert transform to $\mathcal{S}_0$ can be made to be much more direct and we therefore manage to only use
the Hilbert transform once, yielding our linear estimate from below. {\tmname{Bourgain}}'s argument uses a clever random generator to get the sign tosses of the dyadic random walks. He then increases the frequency of each increment and uses the mapping properties of the Hilbert transform to obtain the control he needs in the strong form. In our argument, we change the random generator to fit our operator, which forces it to have memory - a typical difficulty when dealing with shift operators of positive complexity. We also proceed by increasing the frequency, but use a novel striking similarity between $\mathcal{S}_0$ and the Hilbert transform in the dualized form. 

 We want to mention the use of the so-called (quasi-) periodization in the Haar system, loosely based on high frequency modulation in conjunction with $\mathcal{S}_{\tmop{cl}}$ and the Hilbert transform in the technical paper by {\tmname{Karakroumpas}}--{\tmname{Treil}} \cite{KT}. This work in turn builds in part on ideas of {\tmname{Nazarov}} in his work to disprove the so-called Sarason conjecture on the two weight problem for the Hilbert transform - a problem that turned out extremely difficult, subsequently finalized by {\tmname{Lacey}} \cite{Lacey}. After completion of this text here, we used in our seminal paper with {\tmname{Treil}} and {\tmname{Volberg}} quasi-periodization on the Haar system to transfer a dyadic counter example to the famous matrix $A_2$ conjecture to the Hilbert transform via an operator involving both $\mathcal{S}_{\tmop{cl}}$ and $\mathcal{S}_{0}$ \cite{DPTV}. Notice that in the presence of a weight, the translation invariance of the $L^p$ spaces disappears, which is a crucial feature in the argument by {\tmname{Bourgain}} as well as our argument in this text. However, constructing a counter example to a statement is very different from giving a positive result, which is what our task is here. The classical shift operator cannot be used to get the lower estimate on the norm of the Hilbert transform in the way we obtain it in this paper.

\paragraph{The upper bound.} Theorem \ref{theorem_upper} is interesting in the light that the averaging procedure used in \cite{Pet2000} fails and that the coefficient in the estimate we obtain by our means is 1. 
Indeed, the connection established in {\cite{Pet2000}} between
the Hilbert transform $\mathcal{H}_{\mathbb{R}}$ on the real line and the
classical Haar shift $\mathcal{S}_{\tmop{cl}}$ states that the Hilbert transform is, up to a non-zero universal
multiplicative constant, an average of translated and dilated classical Haar
shifts $\mathcal{S}_{\tmop{cl}}^{\alpha, r}$ of the form
\[ \mathbb{E}^{\alpha} \mathbb{E}^r  \mathcal{S}_{\tmop{cl}}^{\alpha, r} =c
   \mathcal{H}_{\mathbb{R}} , \; \; c\neq 0. \]
Here $\alpha$ denotes the translation parameter, $r$ the dilation parameter,
$\mathbb{E}^{\alpha}$ the averaging operator with respect to translations (a limiting procedure was involved), and
$\mathbb{E}^r$ the averaging operator with respect to dilations (logarithmic averages).
The upper--bound $\|
\mathcal{H}_{\mathbb{R}} \|_{L^p_X \rightarrow L^p_X} \lesssim \|
\mathcal{S}_{\tmop{cl}} \|_{L^p_X \rightarrow L^p_X}$ follows. Unfortunately, we prove in Section \ref{S: averaging dyadic Hilbert
transform} that this strategy does not work for
$\Sha$ since this averaging procedure yields 
\[ \mathbb{E}^{\alpha} \mathbb{E}^r  \Sha^{\alpha, r} = 0. \]

\
However, the comparison of the two operators can be achieved through
stochastic representations of the Hilbert transform. It will reveal a profound
connection between $\mathcal{H}$ and $\Sha$, which was an important
motivation for us for bringing a `dyadic Hilbert transform' into play. It is well known and an easy 
consequence of the Cauchy--Riemann equations and the Ito integral formula that the Hilbert transform 
has a representation through the use of harmonic functions and two--dimensional Brownian motion in the disc via stochastic integration. In our construction we approximate these two one--dimensional Brownian motions
by extracting two specific well chosen one--dimensional discrete random walks out of the dyadic tree.
These discrete random walks have memory and the operator $\Sha$ forces this feature. Via methods in stochastic numerical analysis we show a convergence 
between sampled discrete martingales and those governed by Brownian motion. The upper estimate in Theorem \ref{theorem_upper} follows. 
This approach is very different from all representation theorems of singular operators via dyadic shifts as the estimate is obtained via a convergence of a numerical model and not via a representation and convexity.

\

Even if the proofs for the upper and lower bound appear to be very different, they have an ideological common theme: an explanation of the stochastic integral representation for $\mathcal{H}$ in a dyadic world.

\

\paragraph{Quantative estimates when $X=\mathbb{R}$.}

Finally, it is worth noting that the estimate $s_p \le c_0^{-1} h_p$ has new implications even in the real valued case. Indeed, when $X=\mathbb{R}$ or $X=\mathbb{T}$ there hold 
\[
h_p= \left\{ \begin{array}{ll} \tan \frac{\pi}{2p}& 1<p\le 2\\ \cot  \frac{\pi}{2p}& 2<p<\infty \end{array}\right.,
\qquad 
m_p= \left\{ \begin{array}{ll} (p-1)^{-1}& 1<p\le 2\\ p-1& 2<p<\infty \end{array}\right..
\]

$\Sha$ inherits a multiple of the estimate for $h_p$. The constant $c_0$ can be written down as an integral
 and expressed as a multiple of the Catalan constant $G$, that can be evaluated numerically. We refer to Section \ref{S: Constant_c0} for the details. It follows from the explicit expression of $c_0$ that we can compare
\[
	h_p \leq s_p  \leq 1.34689 \ h_p.
\]

The estimate of the exact norm of the Hilbert transform uses orthogonality and to the best of our knowledge,
there has not been previously any way to incorporate orthogonality into dyadic estimates.
Known estimates for shift operators -- also for odd ones -- are thus large multiples of $m_p$ usually depending upon the complexity of the shift operator.
On the contrary, our estimates yield for large $p$, where $m_p \sim p$, the estimate
\[
s_p  \leq c_0^{-1}\ h_p \sim \frac{\pi}{4G} p \leq 0.85746 \ m_p.
\]
We think that finding an estimate proportional to $h_p$ in the dyadic setting is remarkable.
Indeed, estimates for discrete operators can be decisively more difficult than estimates for their continuous counterparts.
We cite in this direction the impressive sharp estimate of  {\tmname{Ba{\~n}uelos--Kwa\'{s}nicki}} \cite{BK} for the Hilbert transform on integers. In its predecessor, our paper {\cite{DomPet2014}}, we proved such sharp estimates for the squares of discrete Riesz transforms resembling for example the continuous operator $\mathcal{R}^2_1-\mathcal{R}^2_2$. Again, due in part to the evenness of the operator under consideration, the solution in {\cite{DomPet2014}} is a lot simpler than that for the discrete Hilbert transform in {\cite{BK}}.

\section{Definitions and statements of the main results}
We will consider functions on domains $\mathbb{H}=\mathbb{R}$ or $\mathbb{H}=\mathbb{T}$ where $\mathbb{T}=[-\pi,\pi)$ is the one dimensional torus $\mathbb{T}=\partial \mathbb{D}$, the boundary of the unit disc $\mathbb{D}$. The Hilbert transform on $\mathbb{R}$ respectively on $\mathbb{T}$ is
\[
\mathcal{H}_{\mathbb{R}}f(x)=p.v.\frac1{\pi}\int_{\mathbb{R}}\frac{f(t)}{x-t}\mathd t, \qquad
\mathcal{H}_{\mathbb{T}}f(x)=p.v.\frac1{2\pi}\int_{\mathbb{T}}f(t)\cot\big(\frac{x-t}{2}\big)\mathd t.
\]

\

If $X$ is a Banach space and $1<p<\infty$ then the Bochner-Lebesgue space $L^p_X(\mathbb{H})$ consists of all strongly measurable functions $f:\mathbb{H}\to X$ such that 
\[
	\| f \|_{L^p_X (\mathbb{H})} \assign \Big( \int_{
  	 \mathbb{H}}  | f (x)
	|_X^p \mathd x \Big)^{1 / p} < \infty .
\]
Recall that $f$ is called weakly measurable if the map $t\mapsto \langle f(t),x^*\rangle$ is measurable for all $x^*\in X^*$ and separably valued if there exists a separable subspace $X_0\le X$ such that $f(t)\in X_0$ a.s. and finally $f$ is called strongly measurable if it is both separably valued and weakly measurable.  

If $f$ real valued and $f\in C_c^1(\mathbb{H})$ and $x\in X$ then write $f\otimes x : \mathbb{H}\to X,(f\otimes x)(t)=f(t)x$ and note the set of all finite linear combinations of this type $C_c^1(\mathbb{H}) \otimes X$, producing a dense subspace of $L^p_X(\mathbb{H})$. This allows us to densely define an operator $T$ originally defined on real valued functions on $X$ valued functions by defining $T$ on $f=\sum_{k=1}^N f_k \otimes x_k \in C_c^1(\mathbb{H}) \otimes X$ by 
\[
Tf(t)=\sum_{k=1}^N Tf_k(t)x_k \in X.
\]
We will use this definition without mention throughout the text.
If $X$ is reflexive then $L^p_X(\mathbb{H})^*=L^{p'}_{X^*}(\mathbb{H})$, where as usual $1/p+1/p'=1$. 

\

The dyadic filtration $(\mathcal{F}_n)_{n\ge 0}$ on $\Omega=\{-1,1\}^\mathbb{N}$ endowed with the probability measure $\mathbb{P}=\otimes (\delta_1+\delta_{-1})/2$ is the filtration associated to the sequence of coordinate functions $\varepsilon_n: \Omega\to \{-1,1\}$. Thus, we set $\mathcal{F}_n=\sigma(\varepsilon_0,...,\varepsilon_{n-1})$ for $n\ge 1$ and $\mathcal{F}_0=\{\emptyset, \Omega \}$. The variables $\varepsilon_n$ are independent and take the values $\pm 1$ with equal probability. Note that $\mathcal{F}_n$ admits exactly $2^n$ atoms. 

An $X-$valued martingale $f_n:\Omega \to X$ adapted to $(\mathcal{F}_n)$ is characterized by the property that $\forall n\ge 1, (f_n-f_{n-1})(\varepsilon_0,...,\varepsilon_{n-1})=\varepsilon_{n-1}d_{n-1}f(\varepsilon_0,...,\varepsilon_{n-2}),$ where it is implicit that $d_{n-1}f$ depends only upon $\varepsilon_1,...,\varepsilon_{n-1}$.

The classical definition of the UMD property, standing for \emph{Unconditional Martingale Difference}, is the following:

\begin{definition*}
A Banach space $X$ is said to have the UMD--p property if there exists a constant $C_p$ so
that for any sign tosses $\alpha_k = \pm 1$, $k\in\mathbb{N}$,
for any martingale $f_n$, there holds
\[ \big\| \sum_{k\geqslant 0} \alpha_{k} \mathd_k f (\varepsilon_0, \ldots,
   \varepsilon_{k-1}) \varepsilon_{k} \big\|_{L_X^p} \leqslant C_p \big\|
   \sum_{k\geqslant 0} \mathd_k f (\varepsilon_0, \ldots, \varepsilon_{k-1}) \varepsilon_{k} \big\|_{L_X^p} . \]
The best such constant for $L^p$ is called the UMD--$p$ constant of $X$ and denoted here $m_p$. 
\end{definition*}

The Walsh functions are for any finite $A\subset \mathbb{N}$, $\omega_A=\Pi_{n\in A}\varepsilon_n$ with 
$\omega_{\emptyset}=1$. $\{ \omega_A \mid  A\subset \{1,...,n\} \}$ resp. $\{\omega_A \mid A {\text{ finite}} \}$ is an orthonormal basis of $L^2(\Omega,\mathcal{F}_n,\mathbb{P})$ resp. $L^2(\Omega,\mathcal{F}_\infty,\mathbb{P})$.

In analysis it is customary to write $r_n$ for the Rademacher functions instead of $\varepsilon_n$. There, $(r_n)_{n\ge 1}$ are defined on the interval $I_0=[0,1)$ endowed with Lebesgue measure by $r_n(x)=\tmop{sign}\sin(2^n\pi x)$ for $n\ge 1$. The sequence $(r_n)$ has the same distribution on $I_0$ as the sequence $(\varepsilon_n)$ on $(\Omega, \mathbb{P})$. Let $\mathcal{F}_n=\sigma(r_1,...,r_n)$, then $\mathcal{F}_n$ is generated by the $2^n$ atoms 
$$\mathcal{D}_n=\left\{ [m 2^{- n}, (m + 1) 2^{- n}) : 0\le m \le 2^n-1 \right\}.$$ 
$\mathcal{F}_n$ is also the sigma algebra generated by the first $2^{n}$ Haar functions $h_I=|I|^{-1/2}(\tmmathbf{1}_{I_+}-\tmmathbf{1}_{I_-})$ when counting as follows: Let $I_n^m=[m 2^{-n}, (m+1) 2^{-n})$ and set $h_0=1$ and count $h_{2^n+k}=h_{I_n^k}$ for $n\ge 0$ and $0\le k\le 2^n-1$.

The dyadic system over the interval $I_0$ can be extended an orthonormal system on $\mathbb{R}$ in many ways, one of which is this. The set of dyadic intervals of generation (or time) $k$ is $\mathcal{D}_k$ on $\mathbb{R}$ is
$ 
\mathcal{D}_k \assign \left\{ [m 2^{- k}, (m + 1) 2^{- k}) : m\in\mathbb{Z} \right\}  
$ 
so that $\mathcal{D}=\bigcup_{k\in\mathbb{Z}} \mathcal{D}_k$.
We note further $\mathcal{D}^- \subset
\mathcal{D}$ the subset of left children and $\mathcal{D}^+ \subset
\mathcal{D}$ the subset of right children in the dyadic grid $\mathcal{D}$
as well as $\mathcal{D}_k^-
\assign \mathcal{D}_k \cap \mathcal{D}^-$ the left children if
$\mathcal{D}_k$, and $\mathcal{D}_k^+ \assign \mathcal{D}_k \cap
\mathcal{D}^+$ the right children in $\mathcal{D}_k$. 

Now, if $I\in\mathcal{D}$ is a given dyadic interval,
we note $\mathcal{D}(I)\assign\{J\in\mathcal{D}: J\subset I \}$,
and $\mathcal{D}^\pm(I)\assign\{J\in\mathcal{D}^\pm: J\subsetneq I \}$.
For example with $I_0\assign[0,1)$, we have $\mathcal{D}(I_0) = \{
I_0 \} \cup \mathcal{D}^-(I_0) \cup \mathcal{D}^+(I_0)$.

The orthonormal basis expansion of a $L^2$ function $f:\mathbb{R}\to\mathbb{R}$ is:
\begin{equation}
	f (x) = \sum_{I \in \mathcal{D}} (f, h_I) h_I (x)
			= \sum_{k\in\mathbb{Z}} \sum_{I \in \mathcal{D}_k} (f, h_I) h_I (x),
\label{eq: Haar decomposition on R}
\end{equation}
with $(f,g)=\int_{\mathbb{R}} f(x) g(x) \mathd x$
. If $f$ has values in a Banach space $X$ and is strongly measurable, then $(f, h_I) \in X$, defined in the usual way as was described above.

\

Unless otherwise specified, we consider Haar expansion
for functions $f$ with support on $I_0=[0,1)$ and therefore write
\[ 
	f (x) = \langle f \rangle_{I_0} + \sum_{I \in \mathcal{D}(I_0)} (f, h_I) h_I (x), \qquad \tmmathbf{1}_I\langle f \rangle_{I} = \langle f \rangle_{I_0} +  \tmmathbf{1}_I\sum_{J \supsetneq I} (f, h_J) h_J (I),
\]
where we noted $h_J (I)=|I|^{-1}\int_I h_J(t)\mathd t = \pm 1/\sqrt{|J|}$ for $I\in J_\pm$ and 
where the latter holds for
 $I\in\mathcal{D}(I_0)$ and is  the so-called MRA property.


As is classical, the Haar series
deliver the martingale difference sequence $f_n$ of the martingale $(f_n)_{n\geqslant 0}$:

\[
f_n(x) 	= \mathbb{E}\left( f \vert \mathcal{F}_n \right)
		\assign \sum_{I \in \mathcal{D}_n} \langle f \rangle_I \tmmathbf{1}_I(x).
\]
This defines the dyadic martingale on the dyadic filtration $(\mathcal{F}_n)_{n\geqslant 0}$,
with the martingale property
$
\mathbb{E}\left( f_m \vert \mathcal{F}_n \right) = f_n(x)
$
for any $m,n$ with $m\geqslant n$.
One way to link back to the sign tosses is by arrival:
\[
	\varepsilon_I(x)=(\tmmathbf{1}_{I_+} - \tmmathbf{1}_{I_-}) (x)=|I|^{1/2} h_I(x)
\]
so that $\varepsilon_I\in\{-1,+1\}$, and setting
\[
	\mathd_I f	\assign  \frac{1}{2}  (\langle f\rangle_{I_+} - \langle f \rangle_{I_-})
   				= \langle f\rangle_{I_+} - \langle f \rangle_{I}
   				= - (\langle f\rangle_{I_-} - \langle f \rangle_{I}),
\]
we can rewrite the Haar series as
\[
	f (x) = \langle f \rangle_{I_0}  + \sum_{I \in \mathcal{D}} \mathd_I f\ \varepsilon_I(x),
	\quad
	\tmmathbf{1}_I \langle f \rangle_{I} = \langle f \rangle_{I_0}  + \tmmathbf{1}_I\sum_{J \supsetneq I} \mathd_J f\ \varepsilon_J(I),
\]
where we noted $\varepsilon_J (I)=\pm 1$ for $I\subset J_\pm$.
Introducing the two random variables
\[
	\mathd_n f(x) = \sum_{I \in \mathcal{D}_n} \mathd_I f\ \tmmathbf{1}_I(x),
	\quad
	\varepsilon_n(x) = \sum_{I \in \mathcal{D}_n} \varepsilon_I(x).
\]
and using the expression of $\langle f \rangle_I$ above allows us to rewrite
\[
	f_n(x)
	= \langle f \rangle_{I_0}  + \sum_{k=0}^{n-1} \mathd_k f(x)\ \varepsilon_k(x)
	= \langle f \rangle_{I_0}  + \sum_{k=0}^{n-1} \mathd_k f(\varepsilon_0(x),\ldots,\varepsilon_{k-1}(x)) \ \varepsilon_k(x).
\]

\subsection{Results}

The dyadic shift of $\mathbb{R}$ is densely defined on the Haar system on $\mathbb{R}$ by
\[\mathcal{S}_0:h_{I_\pm}\mapsto \pm h_{I_\mp}.\]
When defining the dyadic shift on $I_0$, we set
$\Sha: h_{I_{\pm}}\to\pm h_{I_{\mp}},\ \forall I\subseteq I_0$, together with $\Sha(\mathbf{1}_J)=0$, $\Sha(h_J)=0$.
Which operator is used will be clear from the context.

\

The following are our main results. The first theorem corresponds to {\tmname{Burkholder}}'s direction and is a \emph{linear} upper estimate for the Hilbert transform with coefficient 1:

\begin{theorem}\label{theorem_upper}
Let $X$ be a UMD Banach space.
We have
  \[ \| \mathcal{H} \|_{L^p_X \rightarrow L^p_X}
  \leqslant
  \| \Sha \|_{L^p_X \rightarrow L^p_X} . \]
\end{theorem}
It is further our task to prove the lower bound, namely that {\tmname{Bourgain}}'s direction holds with a \emph{linear} relation of the norms:

\begin{theorem}\label{theorem_lower}
Let $G$ denote the Catalan constant and set $c_0=8 G / \pi^2$. Then for any UMD Banach space $X$, we have
\[
	\| \Sha \|_{L_X^p \rightarrow L_X^p} \leqslant c_0^{-1} \| \mathcal{H}
    \|_{L_X^p \rightarrow L_X^p} .
\]
\end{theorem}

Here $c_0 = 8 G / \pi^2 \sim 0.742454$, and $c_0^{-1} \sim 1.34689$.

\

We remark that we can state the norm relations above using the Hilbert transform on $\mathbb{R}$ or $\mathbb{T}$ as the norms are identical via the use of the transference method, detailed in \cite{GeiMonSak2010}. We prove in this text that it is indifferent whether we estimate $\mathcal{S}_0$ defined on $\mathbb{R}$ or on $I_0$.

\

Recall that in the Haar basis notation the definition of the UMD--$p$ constant is
$ m_p = \sup_{\alpha}\| \mathcal{T}_{\alpha} \|_{L_X^p \mapsto L_X^p}$
where $\mathcal{T}_{\alpha}$ denotes the $\pm 1$ martingale multiplier.
In this language the UMD conjecture reads as follows:

\begin{conjecture}[UMD Conjecture] There exist constants $c,C>0$ such that for any Banach space $X$, there holds
\[
	c \sup_{\alpha}\| \mathcal{T}_{\alpha} \|_{L_X^p \mapsto L_X^p} \leqslant 
	 \| \mathcal{H}\|_{L_X^p \rightarrow L_X^p} \leqslant 
	C \sup_{\alpha}\| \mathcal{T}_{\alpha} \|_{L_X^p \mapsto L_X^p}.
\]
\end{conjecture}

Our results thus prove the desired estimates for the pair $\Sha$ and $\mathcal{H}$ instead of the pair $\mathcal{T}_{\alpha}$ and $\mathcal{H}$. Also importantly, it reduces the UMD conjecture to the pair $\mathcal{T}_{\alpha}$ and $\Sha$, both simple dyadic operators.

\section{Tuning of the dyadic filtration}\label{S: Tuning}

For our problem, it is crucial to adapt the dyadic representation to the action of $\Sha$.
First rewrite the Haar series by emphasizing the pairs of sibling intervals $I_\pm$ as
\begin{equation}
	f = \langle f \rangle_{I_0} + (f, h_{I_0}) h_{I_0} + \sum_{k = 0}^{\infty}
   		\sum_{I : |I| = 2^{- k} |I_0 |} (f, h_{I_-}) h_{I_-} + (f, h_{I_+}) h_{I_+}.
\label{eq: Haar decomposition bis}
\end{equation}
This suggests introducing left and right tosses,
respectively $(\varepsilon_k^-)_{k\geq 1}$ and $(\varepsilon_k^+)_{k\geq 1}$, defined as
\[
	\varepsilon^\pm_k(x) = \sum_{I \subset \mathcal{D}^\pm_k} \varepsilon_I(x) = \sum_{I \subset \mathcal{D}^\pm_k} |I|^{1/2} h_I (x).
\]
It follows that for all $k\geq 1$, all $x\in I_0$, we have $\varepsilon_k(x) = \varepsilon_k^-(x) + \varepsilon_k^+(x)$,
with the left and right tosses having disjoint supports forming a partition of $I_0$.
For $k=0$, we have no left or right toss and therefore we use the previously defined toss $\varepsilon_0(x)$.
The Haar expansion \eqref{eq: Haar decomposition bis} can now be written as
\begin{eqnarray*}
	f	& =	& \langle f \rangle_{I_0} + \mathd_0 f\ \varepsilon_0
		+\ \sum_{k = 1}^{\infty}
		\mathd_k f(\varepsilon_0,\varepsilon^-_{1}+\varepsilon^+_{1},\ldots,\varepsilon^-_{k-1}+\varepsilon^+_{k-1})\ (\varepsilon^-_k + \varepsilon^+_k).
\end{eqnarray*}
$\varepsilon^-_k \neq 0$ iff $\varepsilon^-_{k-1}+\varepsilon^+_{k-1}=-1$ and 
 $\varepsilon^+_k \neq 0$ iff $\varepsilon^-_{k-1}+\varepsilon^+_{k-1}=+1$. To remind of this we introduce the notation 
 \begin{eqnarray*}
 \lefteqn{\mathd_k f^{\pm}(\varepsilon_0,\varepsilon^-_{1},\varepsilon^+_{1},\ldots,\varepsilon^-_{k-1},\varepsilon^+_{k-1})}\\
 &=&
 \left\{\begin{array}{ll}
 \mathd_k f(\varepsilon_0,\varepsilon^-_{1}+\varepsilon^+_{1},\ldots,\varepsilon^-_{k-1}+\varepsilon^+_{k-1}) & \text{if } \varepsilon^-_{k-1}+\varepsilon^+_{k-1}=\pm 1\\
 0 &\text{if } \varepsilon^-_{k-1}+\varepsilon^+_{k-1}=\mp 1 \\
  \end{array}\right. .
 \end{eqnarray*} 

The action of $\Sha$ reads
\begin{align*}\MoveEqLeft
	\Sha f  =  0\ +\ 0
		+ \ \sum_{k = 1}^{\infty}
		\ \mathd_k f^+(\varepsilon_0,\varepsilon^-_{1},\varepsilon^+_{1},\ldots,\varepsilon^-_{k-1},\varepsilon^+_{k-1})\ \varepsilon^-_k
		\\
		 & \hspace{2.5cm} - \mathd_k f^-(\varepsilon_0,\varepsilon^-_{1},\varepsilon^+_{1},\ldots,\varepsilon^-_{k-1},\varepsilon^+_{k-1})\ \varepsilon^+_k.
\end{align*}

The change of variable above is straightforward since for all $k\geq 1$, all $x\in I_0$, we have $\varepsilon_k(x) = \varepsilon_k^-(x) + \varepsilon_k^+(x)$.
It is further noticeable that for any given $x\in I_0$, only one term is present in each summand.
To repeat, if $\varepsilon^-_{k-1}= -1$ or
$\varepsilon^+_{k-1}= -1$ (i.e. $\varepsilon_{k-1}= -1$), then $x\in\mathcal{D}^-_k$, hence $\varepsilon_k^-(x)\neq 0$
and $\varepsilon_k^+(x)=0$. Similarly if $\varepsilon^-_{k-1}= +1$ or
$\varepsilon^+_{k-1}= +1$ (i.e. $\varepsilon_{k-1}= +1$), then $x\in\mathcal{D}^+_k$, hence $\varepsilon_k^-(x)= 0$
and $\varepsilon_k^+(x)\neq 0$. See Figure \ref{fig: dyadic tosses} below for an illustration.
\begin{figure}[ht]
\centering
\begin{subfigure}{.5\textwidth}
  \centering
  \includegraphics[width=0.9\linewidth]{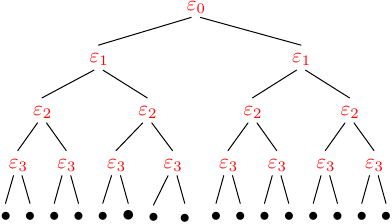}
\end{subfigure}%
\begin{subfigure}{.5\textwidth}
  \centering
  \includegraphics[width=0.9\linewidth]{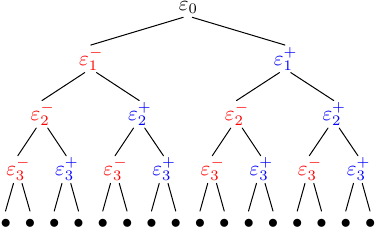}
\end{subfigure}
\caption{Standard dyadic tosses (left) versus dyadic tosses adapted to $\Sha$ (right)}
\label{fig: dyadic tosses}
\end{figure}

\section{First $L^p$ estimates} \label{S: first Lp estimates}

We have considered so far $\Sha$ acting on either functions defined on $\mathbb{R}$ or functions defined on a $I_0$. In this section, we note specifically $\Sha^{\mathbb{R}}: h_{I_{\pm}}\to\pm h_{I_{\mp}}, \forall I\in\mathcal{D}$. Further, let us denote $\Sha^J$ the operator acting on functions defined on $J$, defined by
$\Sha^J: h_{I_{\pm}}\to\pm h_{I_{\mp}},\ \forall I\subseteq J$, $\Sha^J(\mathbf{1}_J)=0$, $\Sha^J(h_J)=0$.

We claim that all those operators have the same $L^p$ norm

\begin{theorem}\label{T: comparison Lp norms}
Let $J\in\mathcal{D}$ any dyadic interval. We have
\[
\big\| \Sha^J \big\|_{L^p_X(J) \rightarrow L^p_X(J)}
=
\big\| \Sha^{\mathbb{R}} \big\|_{L^p_X(\mathbb{R}) \rightarrow L^p_X(\mathbb{R})}
\]
\end{theorem}

\begin{proof}
Without loss of generality, we consider the case $J=I_0$. Let $f$ defined on $I_0$ and $\tilde f$ its extension to $\mathbb{R}$ with $\tilde f = 0$ on $\mathbb{R}\backslash I_0$.
The Haar decomposition of $\tilde f$ on $\mathbb{R}$ involves only dyadic intervals
$I\supset I_0$ or dyadic intervals $I\subsetneq I_0$. In the case $I\subsetneq I_0$, we have for the sibling $I^\prime$ of $I$ that also $I^\prime \subsetneq I_0$, hence
$\Sha^{\mathbb{R}} h_I = \Sha^{I_0} h_I$. On the converse when $I\supset I_0$ we have 
$I^\prime \cap I_0 = \emptyset$. It follows that $\Sha^{I_0} f$ is the restriction to $I_0$
of $\Sha^{\mathbb{R}} \tilde f$. This implies
\[
\big\| \Sha^{I_0} f \big\|_{L^p_X(I_0)}
	\leqslant
	\big\| \Sha^{\mathbb{R}} \tilde f \big\|_{L^p_X(\mathbb{R})}
	\leqslant
	\big\| \Sha^{\mathbb{R}} \big\|_{L^p_X(\mathbb{R}) \rightarrow L^p_X(\mathbb{R})}
			\| \tilde f \|_{L^p_X(\mathbb{R})}.
\]
With $\| \tilde f \|_{L^p_X(\mathbb{R})} = \| f \|_{L^p_X(I_0)}$ this proves the claim for $I_0$.

\

On the converse, let now $f$ defined on $\mathbb{R}$. We assume without loss of generality that $f$ is bounded with compact support on $\mathbb{R}^+$. Let $J\supset I_0$ a large dyadic interval containing the support of $f$. We split
$\Sha^{\mathbb{R}} = \Sha^{J} + (\Sha^{\mathbb{R}}-\Sha^{J})$ and observe that the difference only involves Haar functions corresponding to the increasing sequence of dyadic intervals $K\supset J$. Their transforms $\Sha^{\mathbb{R}} h_K$ involve dyadic intervals forming a disjoint family. Each corresponding term in the dyadic decomposition of $f$ or its transform is of order at most $ \| f \|_{L^1_X(J)} / |K|  $, yielding
\[
\big\| \Sha^{\mathbb{R}} f \big\|_{L^p_X({\mathbb{R}\backslash J)}}
\lesssim
\Big( \sum_{K\supset J}\big( |K|^{-1}\|f \|_{L^1_X(J)}  \big)^p \Big)^{1/p}
\lesssim
|J|^{-1}\|f \|_{L^1_X(J)}.
\]
It is further not difficult to prove using simple scaling arguments that for any dyadic interval $J$, we have 
$\| \Sha^{J} \|_{L^p_X(J)} = \| \Sha^{I_0} \|_{L^p_X(I_0)}$,
therefore also
\[
\big\| \Sha^{\mathbb{R}} f \big\|_{L^p_X(\mathbb{R})}
\leqslant
\big\| \Sha^{I_0} \big\|_{L^p_X(I_0)} \| f \|_{L^p_X(\mathbb{R})} + C |J|^{-1}\|f \|_{L^1_X(\mathbb{R})}
\]
for some constant $C>0$. Passing to the limit $|J|\to\infty$ implies the claim.
\end{proof}

As a conclusion, we can use $s_p$ to denote the norm of either $\Sha^\mathbb{R}$ or $\Sha^{I_0}$.

\section{Proof of the Upper Bound (Theorem \ref{theorem_upper})}


\paragraph{Stochastic representation of the Hilbert transform}
Let $f \in L^p(\mathbb{T})$ given and $g = \mathcal{H} f$ the Hilbert transform of $f$. We use the same notation for their harmonic extensions $f,g \in L^p
(\mathbb{D})$. Let $W$ the standard two--dimensional Brownian motion started
at the origin. For all times $t$ such that
$(W_s)_{0 \leqslant s \leqslant t}$ remains in the unit disc, the It{\^o} formula ensures that a.s.,
\[ f (W_t) = f (0, 0) + \int_0^t \nabla f (W_s) \cdot \mathd W_s, 
\]
since $f$ is harmonic.
Let $\tau$ denote the hitting time of $\partial \mathbb{D}$, 
$  \tau \assign \inf \{ t > 0 ; W_t \nin \mathbb{D} \} . $ 
We have again thanks to It{\^o}'s
formula, if $W_{\tau} = z$, i.e. the random walk hits $\partial
\mathbb{D}$ at $z$,
 $f(z)=f (W_{\tau}) = f (0, 0) + \int_0^{\tau} \nabla f (W_s) \cdot \mathd W_s .$ 
Similarly, owing to the Cauchy--Riemann relations for the conjugate function
$g$ of $f$, 
\[ \forall t \leqslant \tau, \quad g (W_t) = \int_0^t \nabla g (W_s) \cdot
   \mathd W_s = \int_0^t \nabla f (W_s)^{\perp} \cdot \mathd W_s, \]
where
\[ \nabla^{\perp} f (W_s) = \left(\begin{array}{cc}
     0 & - 1\\
     1 & 0
   \end{array}\right) \nabla f (W_s) \]
denotes the vector anticlockwise orthogonal to $\nabla f (W_s)$. In other words,
$g (W_t)$ is a martingale transform of $f (W_t)$ with predictable multiplier
the rotation matrix above. 
In the next sections, we will be dealing with the martingale $\mathcal{M}^f_t
\assign f (W_t)$ and its martingale transform $\mathcal{M}_t^g \assign g
(W_t)$ defined as
\[ 
  \forall t \leqslant \tau, \quad \mathcal{M}_t^f \assign f (0, 0) + \int_0^t
  \nabla f (W_s) \cdot \mathd W_s, \quad \mathcal{M}_t^g \assign \int_0^t
  \nabla^{\perp} f (W_s) \cdot \mathd W_s . 
  \]
Notice that in the case where $f$ and therefore $g$ are real or Hilbert space valued, the situation is very different from the Banach space valued case. The martingales
above are on the one hand differentially subordinate using the norm, i.e.
\[ \mathd [\mathcal{M}^f, \mathcal{M}^f]_t = \| \nabla f (W_s) \|^2 \mathd t =
   \| \nabla f (W_s)^{\perp} \|^2 \mathd t = \mathd [M^g, M^g]_t, \]
and on the other hand are orthogonal, i.e
\[ \mathd [\mathcal{M}^f, \mathcal{M}^g]_t = (\nabla f (W_s), \nabla^{\perp} f
   (W_s)) \mathd t = 0. \]
They were studied in the work of {\tmname{Ba{\~n}uelos}} and {\tmname{Wang}}
{\cite{BanWan1996}}, where they prove the sharp
martingale $L^p$ bound $\| \mathcal{M}^g \|_p \leqslant c_p \| \mathcal{M}^f
\|_p$ using special functions, and where $c_p \assign \|
\mathcal{H}_{\mathbb{R}} \|_{p \rightarrow p}$.

\paragraph{Strategy and motivation}Due to the discrete nature of the dyadic
shift, we first aim at approximating the two stochastic integrals $\mathcal{M}_t^f$ and $\mathcal{M}_t^g$ by using two--dimensional discrete random walks $(B_k)_{k
\geqslant 0}$ built on top of the dyadic system. We expect discrete
martingales $M_k^f$ and $M_k^g$ that should be discrete counterparts of $\mathcal{M}_t^f$ and $\mathcal{M}_t^g$, namely
\[ 
  M_k^f (x) \assign f (0, 0) + \sum_{l = 1}^k \nabla f (B_{l - 1} (x)) \cdot
  \mathd B_l (x), 
\]
\[ 
  M_k^g (x) \assign \sum_{l = 1}^k \nabla^{\perp} f
  (B_{l - 1} (x)) \cdot \mathd B_l (x), 
 \] 
where $x$ spans the dyadic probability space $\Omega \assign [0, 1)$ with
uniform probability density $\mathd \mathbb{P} (x) = \mathd x$. In order to
relate the dyadic operator $\Sha$ to the Hilbert transform  $\mathcal{H}$, the
discrete random walk $(B_k)_{k \geqslant 0}$ has to be crafted in a very
specific manner, as we discuss now. We want to build this discrete random walk
so as to obtain the following diagram, where ``convergence'' is meant in a
weak sense defined later.
\[ \begin{array}{lllll}
     & f & \overset{\mathcal{H}}{\longrightarrow} & g & \\
     (\tmop{discretization} / \tmop{convergence}) & \longdownarrow
     \longuparrow &  & \longdownarrow \longuparrow & (\tmop{discretization} /
     \tmop{convergence})\\
     & M_k^f & \overset{\Sha}{\longrightarrow} & M_k^g & 
   \end{array} . \]
With the very specific process $B$ to be chosen, we want $M^g$
to be a martingale transform of $M^f$ through the action of $\Sha$,
namely $M^g = \Sha M^f$, where
\begin{equation}
  (\Sha M^f)_k \assign \sum_{l = 1}^k \nabla f (B_{l - 1} (x)) \cdot
  \Sha \mathd B_l (x) . \label{eq:SMf}
\end{equation}
Notice that $\Sha$ above acts on the increment $\mathd B_l(x)$ and not on the
predictable multiplier $\nabla f (B_{l - 1} (x))$, as one usually expects for martingale transforms.
But rewriting $M^g$ as
\begin{equation}
  M_k^g (x) = \sum_{l = 1}^k \nabla^{\perp} f (B_{l - 1} (x)) \cdot \mathd B_l
  (x) = \sum_{l = 1}^k \nabla f (B_{l - 1} (x)) \cdot \mathd B_l^{\top} (x),
  \label{eq:Mg bis}
\end{equation}
where $\mathd B^{\top}_k$denotes the vector clockwise orthogonal to
$\mathd B_k$,
and comparing \eqref{eq:SMf} with \eqref{eq:Mg bis}, suggests
that the discrete random walk should obey the property
\begin{equation}
  \Sha \mathd B_l {= \mathd B^{\top}_l}  . \label{eq:SdB}
\end{equation}
We construct explicitly such a dyadic random walk in Section \ref{SS: random walks} below.
After that, it will essentially remain to prove the convergence stated in the previous diagram.

We refer to Remark \ref{R: probabilistic interpretation Sha} for an interpretation of $\Sha$ acting on the probability space.

\

\subsection{ Random walks, stopping times and scalings}\label{SS: random walks}

\paragraph{Construction of the discrete random walk}
Let $\delta > 0$. We build a two dimensional random walk $B_k (x) = (B_k^1
(x), B_k^2 (x))$ defined as
\[ 
B^1_k (x) \assign \displaystyle \sum_{l = 1}^k \sqrt{2 \delta}\ \varepsilon_l^+ (x), \qquad
B^2_k (x) \assign \displaystyle \sum_{l = 1}^k \sqrt{2 \delta}\ \varepsilon_l^- (x),
\]
with $B_0 (x) = (0, 0)$ for all $x$. We write 
$\mathd B_l^1 (x)\assign \sqrt{2 \delta}\ \varepsilon_l^+ (x)=\sqrt{2 \delta}\  \tmmathbf{1} (\varepsilon_{l - 1} (x) =
     + 1)\ \varepsilon_l (x)$ 
and 
$\mathd B_l^2 (x) \assign \sqrt{2 \delta}\ \varepsilon_l^- (x)= \sqrt{2 \delta}\ \tmmathbf{1} (\varepsilon_{l - 1} (x) =
     - 1)\ \varepsilon_l (x).$ 
This formulation shows that the random walk $B_k$ is not a Markov process in the dyadic filtration but a
so--called Markov process with memory. See Figure \ref{Fig: DRW} below for an illustration.
\begin{figure}[ht]
\centering
\begin{subfigure}{.5\textwidth}
\centering
\includegraphics[width=0.9\linewidth]
{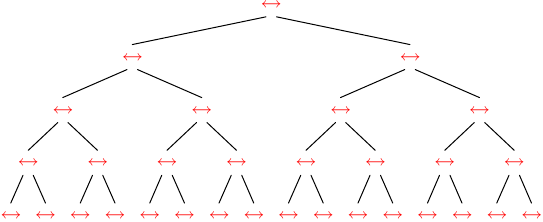}
\end{subfigure}%
\begin{subfigure}{.5\textwidth}
\centering
\includegraphics[width=0.75\linewidth]{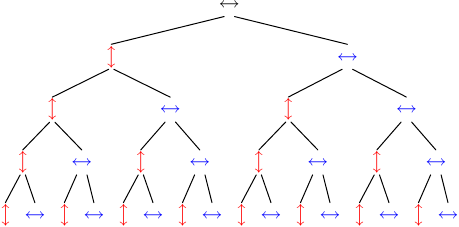}
\end{subfigure}
\caption{Standard one--dimensional discrete random walk on the left. Each toss produces a left or right step. Two--dimensional discrete random walk with memory. Each toss outcome decides the direction, horizontal or vertical, of the next step of the random walk.}
\label{Fig: DRW}
\end{figure}

Since
$$\mathd B^1_l (x)= 
  \sqrt{2 \delta}\ \varepsilon_l^+ (x)
  , \quad \mathd B^2_l (x) = 
  \sqrt{2 \delta}\ \varepsilon_l^- (x),
  $$
 we evaluate now, for any $l\geqslant 1$,
\begin{align*}
  \Sha \mathd B^1_l (x) = & \sqrt{2 \delta} \sum_{I \in
  \mathcal{D}_l^+} |I|^{1/2}  \Sha h_I (x) = \sqrt{2 \delta} \sum_{I
  \in \mathcal{D}_l^-} |I|^{1/2}  (+ h_I (x)) =  + \mathd B^2_l (x)
  \\
  \Sha \mathd B^2_l (x) = & \sqrt{2 \delta} \sum_{I \in
  \mathcal{D}_l^-} |I|^{1/2}  \Sha h_I (x) = \sqrt{2 \delta} \sum_{I
  \in \mathcal{D}_l^+} |I|^{1/2}  (- h_I (x)) = - \mathd B^1_l (x) .
\end{align*}
In other words, denoting $\mathd B_l \assign (\mathd B^1_l, \mathd B^2_l)$ a
two dimensional increment, we have $\Sha \mathd B_l = \mathd B_l^{\top}$
as desired.

We recall now that the discrete martingales associated to $f$ and $g =
\mathcal{H} f$ are
\[ \forall k, \quad M^f_k \assign f (B_0) + \sum_{l = 1}^k \nabla f (B_{l -
   1}) \cdot \mathd B_l, \quad M^g_k \assign g (B_0) + \sum_{l = 1}^k \nabla g
   (B_{l - 1}) \cdot \mathd B_l,
\]
and observe:
\begin{lemma}
  \label{L: Lp estimate for Mng}Let $M^f$ and $M^g$ as above. We have
  $\Sha M^f = M^g$ and therefore also
  \[ \forall k, \quad \big\| M_k^g \big\|_p \leqslant \big\| \Sha \big\|_{p \rightarrow
     p} \big\| M_k^f \big\|_p . \]
\end{lemma}
\begin{proof}
  Since $\Sha \mathd B_l = \mathd B_l^{\top},$ $\nabla g =
  \nabla^{\perp} f$ and $g (B_0) = 0$, we have successively, for all $k
  \geqslant 1$,
  \begin{align*}\MoveEqLeft
    (\Sha M^f)_k  \assign  \Sha \Big( f (B_0) + \sum_{l =
    1}^k \nabla f (B_{l - 1}) \cdot \mathd B_l \Big) \assign 0 + \sum_{l =
    1}^k \nabla f (B_{l - 1}) \cdot \mathcal{S} \mathd B_l\\
     = &  \sum_{l = 1}^k \nabla f (B_{l - 1}) \cdot \mathd B_l^{\top} =
    \sum_{l = 1}^k \nabla f (B_{l - 1})^{\perp} \cdot \mathd B_l
    =  M_k^g .
  \end{align*}
  Hence the result.
\end{proof}

In order to obtain from the inequality above an inequality for the continuous
martingales $\mathcal{M}^f$ and $\mathcal{M}^g$, we need to prove in a
suitable sense the convergence of the discrete martingales $M^f$ and $M^g$
towards their continuous counterparts $\mathcal{M}^f$ and $\mathcal{M}^g$.
Since we are only interested in norms of those processes, what we need to
obtain is some so--called weak convergence estimate, see e.g. {\tmname{Talay}}
or {\tmname{Kloeden}}--{\tmname{Platen}} {\cite{Tal1986n,KloPla1992}}.
Unfortunately, our situation does not exactly fit those expositions, for the
following reasons:
\begin{itemize}
  \item We consider randomly stopped processes as opposed to processes running
  on a prescribed, deterministic, interval of time.
  
  \item The discrete random walk is \tmtextit{not} a Markov process, rather a
  so--called Markov process with memory
  
  \item The quadratic covariations of the discrete process are never
  converging towards that of their continuous counterparts, since they can be
  null half of the time.
  
  \item We do not only compare discrete random walks and their continuous
  counterparts but also their martingale transforms $M^f$ and $M^g$.
\end{itemize}
We define in the next paragraphs some auxiliary processes and suitable stopped
random walks. The reader will observe in the course of the proofs that we need to finely tune the 
different scales involved in order to account for the difficulties above.
At the end of the present section we will be able to state in a precise manner
the convergence result Theorem \ref{T: convergence} suited to our needs.

\paragraph{Scaling and stopped random walks}Let $W \assign (W_t)_{t
\geqslant 0}$ the standard two dimensional Brownian process started at the
origin. Let again $\tau$ the stopping time
$  \tau \assign \inf \{ t > 0 ; W_t \nin \mathbb{D} \}, $ 
that is the first time of exit of the unit disc $\mathbb{D}$. We denote by
$W^{\tau} = (W^{\tau}_t)_{t \in [0, \infty)} \assign (W_{t \wedge \tau})_{t
\in [0, \infty)}$ the corresponding stopped process. Finally, given $t \in
\mathbb{R}$, $x \in \mathbb{D}$, we note $(W^{\tau, t, x}_{t + s})_{s
\geqslant 0}$ the process $W^{\tau}$ started at $x$ at time $t$.

Let $T > 0$ a fixed time, and $N \in \mathbb{N}$ large. Define $\delta$ a
small time--step such that $T = N^5 \delta$. In order to denote the
corresponding discrete times, we will use the indices $k, l \in [0, N^5]$,
typically $t_k \assign k \delta$, $t_l \assign l \delta$. We introduce a
larger time-step $\theta$ defined as $\theta \assign N \delta$. Notice that $T
= N^4 \theta$, so that also $\theta$ tends to zero as $N$ goes to infinity for
fixed $T$. The discrete times corresponding to this larger time--step will use
indices $n, m \in [0, N^4]$, typically $t_n \assign n \theta$, $t_m \assign m
\theta$. Given the discrete random walk $B$ defined above, we define a new
discrete random walk $X$ by sampling $B$ at times that are multiples of
$\theta = N \delta$, therefore with indices that are multiples of $N$,
 $ \forall n \geqslant 0, \quad X_n \assign B_{n N} . $ 
Notice further that
\[ \mathd X_n \assign X_n - X_{n - 1} = B_{n N} - B_{(n - 1) N} = \sum_{l =
   1}^N \mathd B_{(n - 1) N + l} . \]
In order to stop $X$ just before it leaves the unit disc, we set $\varepsilon = 1 /
N$ and define the discrete stopping time
$ \tau_{\varepsilon} \assign \inf \{ t_n ; X_n  \in (1 - \varepsilon)
   \mathbb{D} \}, $ 
with the convention $\tau_\epsilon=+\infty$ if $X_n\in (1 - \varepsilon)
   \mathbb{D}$ for all $n\in[0,N^4]$.
We will denote by $n_{\varepsilon}$ the random index such that
$\tau_{\varepsilon} \assign n_{\varepsilon} \theta$.  From the definition of
$X$ we have
\[ \forall n, \quad | X_n - X_{n - 1} | \assign | \mathd X_n | \leqslant \sum_{l = 1}^N |
   \mathd B_{(n - 1) N + l} | \leqslant N \sqrt{2 \delta} = \sqrt{2 T} N^{- 3
   / 2} . \]
It follows that for $N$ large enough, we have
$| \mathd X_n | \leqslant \varepsilon,\ \forall n\in[0,N^4]$.
In particular, before a stopping time we have $X_{\tau_{\varepsilon} - \theta}
\in (1 - \varepsilon) \mathbb{D}$ which implies that the stopped process
$X^{\tau_{\varepsilon}}$ always remains in the unit disc $\mathbb{D}$.
On the converse, if $\tau_\epsilon < +\infty$, then at stopping time,
the random walk $X_{\tau_{\varepsilon}}$ lies in
the band $\mathbb{D} \backslash (1 - \varepsilon) \mathbb{D}$ of width
$\varepsilon$ near the interior boundary $\partial \mathbb{D}$ of the unit
disc, and the next increments are null.
We note $X^{\tau_{\varepsilon}}$ the corresponding stopped process,
that is $(X_n^{\tau_{\varepsilon}})_{n\in[0,N^4]} \assign (X_{n\wedge n_\epsilon})_{n\in[0,N^4]}$.
Finally we will denote by $k_{\varepsilon}$ the random index such that
$\tau_{\varepsilon} \assign k_{\varepsilon} \delta$, or equivalently
$k_{\varepsilon} \assign N n_{\varepsilon}$. We note $B^{\tau_{\varepsilon}}$
the process $B$ stopped at $\tau_{\varepsilon}$.

Notice that the discrete martingales $(M_k^f)_{k \in [0, N^5]}$ and
$(M_k^g)_{k \in [0, N^5]}$ are martingale transforms of the discrete random
walk $(B_k)_{k \in [0, N^5]}$. We will aslo consider a sampled version of
$(M_k^f)_{k \in [0, N^5]}$, namely $(M_n^f)_{n \in [0, N^4]}$, where $M_n^f
\assign M^f_k$ for $k = n N$. We use the same notation for both processes.
Which one is meant will be clear from the context. Similarly we note
$(\mathcal{F}_k)_{k \in [0, N^5]}$ the filtration associated to the discrete
random walk $(B_k)_{k \in [0, N^5]}$ and $(\widetilde{\mathcal{F}}_n)_{n \in
[0, N^4]}$, with $\widetilde{\mathcal{F}}_n \assign \mathcal{F}_{n N}$, the
filtration associated to the coarse random walk $(X_n)_{n \in [0, N^4]}$.

\begin{remark} \label{R: probabilistic interpretation Sha}
We have defined the discrete dyadic random walks based on the requirement that for any $x\in[0,1)$, any $k\in[0,N^5]$, the increments of the random walk satisfy $(\Sha \mathd B_k)(x) = (\mathd B_k(x))^\top$. Since the random walk starts at $(0,0)$, this is equivalent to $(\Sha B_k)(x) = (B_k(x))^\top$, that is the action of the operator $\Sha$ is a clockwise orthogonal rotation.

We can go further and observe that also for the stopped processes, we have $(\Sha B^{\tau_\epsilon})(x) = (B^{\tau_\epsilon}(x))^\top$. This is thanks to the fact that the discretised grid $\sqrt{2\delta} \mathbb{Z}^2 \cap \mathbb{D}$ on which $B^{\tau_\epsilon}$ lives is invariant under orthogonal rotations.
This allows us to interpret $\Sha$ as acting directly in the probability space $\mathbb{P}\assign[0,1)$ by setting
    $B^{\tau_\epsilon}(\Sha x) \assign  (B^{\tau_\epsilon}(x))^\top, $ 
where $\Sha$ substitutes the stopping intervals of $B^{\tau_\epsilon}$ for those of $(B^{\tau_\epsilon})^\top$. We want to stress the fact that both the proof of Theorem \ref{theorem_upper} and Theorem \ref{theorem_lower} play tricks with the way we access
the probability space. 
\end{remark}

\paragraph{Notation} We assume without loss of generality that $f \in L^p_X
( \mathbb{T})$ and its harmonic extension (also noted $f$) $f \in
L^p_X (\mathbb{D})$ are smooth Frechet differentiable functions, that is $f
\in \mathcal{C}^k$ for all $k \geqslant 0$.

Now, if $f \assign f (x_1, \ldots, x_m)$ is a $X$--valued function of $m$
variables defined on the open set $U \subset \mathbb{R}^m$, we note $\mathD f
\assign (\partial_1 f, \ldots, \partial_m f)$ its derivatives in the Frechet
sense, where $\mathD f : U \times \mathbb{R}^m \rightarrow X$ is continuous.
Further, given an $m$--multiindex $\alpha \assign (\alpha_1, \alpha_2, \ldots,
\alpha_m)$ with $| \alpha | = k$ we note as usual $\mathD^{\alpha} f \assign
\partial_1^{\alpha_1} \ldots \partial_m^{\alpha_m} f$ its partial derivatives
of order $k$ in the Frechet sense, where $\mathD^{\alpha} f : U \times
(\mathbb{R}^m)^k \rightarrow X$ is continuous. Finally, using again
multiindices, monomials of the from $x^{\alpha}$, where $x \assign (x_1,
\ldots, x_m)$, are a shorthand for $x^{\alpha} \assign x_1^{\alpha_1} \ldots
x_m^{\alpha_m}$.

For the convergence results, our main parameters are $\varepsilon > 0$, a
small number, and $T > 0$, a large number. We note $c(\varepsilon)$, $c_T
(\varepsilon)$, $c (T)$  generic
functions with $c(\varepsilon)\to 0$ uniformly in $T$ as $\varepsilon \to 0$, $c_T
(\varepsilon)\to 0$ as $\varepsilon \to 0$ for any fixed $T$, $c (T) \to 0$ as $T\to \infty$
uniformly in $\varepsilon$. It is implicit that those
functions all depend on the fixed function $f$ and its derivatives. Additional
dependences will be mentioned when needed.

\

\paragraph{Convergence result}In order to prove Theorem \ref{theorem_upper}, we need the following convergence result of the discrete martingales
towards their continuous counterparts.

\begin{theorem}[Convergence of $L^p$ norms of martingales]
  \label{T: convergence}Let $f$ as above. We have
  \[ \lim_{T \rightarrow \infty} \lim_{\varepsilon \rightarrow 0} \mathbb{E} |
     M_T^f |^p =\mathbb{E} | \mathcal{M}^f_{\infty} |^p .\]
\end{theorem}

A key ingredient is the notion of weak consistency presented in the next
Section \ref{S: weak consistency}. This allows us to prove auxiliary
convergence results in Section \ref{S: auxiliary convergence results}. Finally
Section \ref{S: main results} is devoted to the proofs of the main results for the upper bound,
namely Theorem \ref{T: convergence} and Theorem \ref{theorem_upper}.

\subsection{Weak consistency and moment estimates.}\label{S: weak consistency}

Convergence results will be obtained by proving the weak consistency of the
(sampled) stopped discrete random walk $X^{\tau_{\varepsilon}}$ with the
continuous stopped two-dimensional Brownian process $W^{\tau}$. Due to the
stopping process, we can not rely on the standard definition of weak
consistency based on discrete and continuous stochastic equations with
prescribed coefficients depending smoothly on the process alone as used in
e.g. {\cite{Tal1986n,KloPla1992}}. The definition adapted to our situation
simply reads:

\begin{definition}
  We say that $X^{\tau_{\varepsilon}}$ is weakly consistent with $W^{\tau}$,
  iff there exists a function $c \assign c (\varepsilon)$ tending to zero when
  $\varepsilon$ goes to zero, such that for all $n$, all $2$--multiindices
  $\alpha \assign (\alpha_1, \alpha_2)$ with $| \alpha | = 2$, there hold
  \[ \Big| \mathbb{E} \big(X^{\tau_{\varepsilon}}_{n + 1} -
     X^{\tau_{\varepsilon}}_n \,|  \widetilde{\mathcal{F}}_n \big) -\mathbb{E} \big(
     W^{\tau, \tau_{\varepsilon}, X_n^{\tau_{\varepsilon}}}_{t_{n + 1}} -
     W_{t_n}^{\tau, \tau_{\varepsilon}, X_n^{\tau_{\varepsilon}}} \;| 
     \widetilde{\mathcal{F}}_n \big) \Big| \leqslant \theta c
     (\varepsilon), \]
  \[ \Big| \mathbb{E} \big((X^{\tau_{\varepsilon}}_{n + 1} -
     X^{\tau_{\varepsilon}}_n)^{\alpha}\; |  \widetilde{\mathcal{F}}_n\big)
     -\mathbb{E} \big( ( W^{\tau, \tau_{\varepsilon},
     X_n^{\tau_{\varepsilon}}}_{t_{n + 1}} - W_{t_n}^{\tau,
     \tau_{\varepsilon}, X_n^{\tau_{\varepsilon}}})^{\alpha} \; | 
     \widetilde{\mathcal{F}}_n \big) \Big| \leqslant \theta c
     (\varepsilon) . \]
\end{definition}
We can now state
\begin{lemma}
  \label{L: weak consistency} The discrete stopped process
  $X^{\tau_{\varepsilon}}$ is weakly consistent with the continuous stopped
  process $W^{\tau}$.
\end{lemma}

\begin{proof}
  The weak consistency is a straightforward consequence of the next two
  moments lemmas.
\end{proof}

\begin{lemma}[Discrete Moments]
  \label{L: discrete moments}Let $X^{\tau_{\varepsilon}}$ as above. There
  exists a function $c \assign c (\varepsilon)$ tending to zero when
  $\varepsilon$ goes to zero, such that for all $n < n_{\varepsilon}$, all
  $2$--multiindeces $\alpha \assign (\alpha_1, \alpha_2)$ with $| \alpha | = 2$,
  there hold
  \[ \mathbb{E} \big( \mathd X^{\tau_{\varepsilon}}_{n + 1} \;|
     {\widetilde{\mathcal{F}}_n}  \big) = 0, \quad \mathbb{E} \big((\mathd X^{\tau_{\varepsilon}}_{n + 1})^{\alpha}\; |
     \widetilde{\mathcal{F}}_n\big) =\tmmathbf{1} (\alpha_1 \neq \alpha_2) \theta
     (1 + c (\varepsilon)). \]
  Moreover for all $p \geqslant 2$, all $i = 1, 2$, there holds $ 
   \mathbb{E} (| \mathd X_{n + 1}^{\tau_{\varepsilon}, i} |^p |
     \widetilde{\mathcal{F}}_n) \lesssim \theta^{p / 2} . $ 
\end{lemma}

Notice that $\tmmathbf{1} (\alpha_1 \neq \alpha_2) = 1$ if $\alpha \in \{ (2,
0), (0, 2) \}$ and $\tmmathbf{1} (\alpha_1 \neq \alpha_2) = 0$ if $\alpha =
(1, 1)$.

\begin{lemma}[Continuous Moments]
  \label{L: continuous moments} Let $W^{\tau}$ and $W^{\tau, t, x}$ as above.
  There exists a function $c \assign c_T (\varepsilon)$ tending to zero when
  $\varepsilon$ goes to zero for all fixed $T$, such that for all
  $2$--multiindeces $\alpha \assign (\alpha_1, \alpha_2)$ with $| \alpha | = 2$,
  there hold
  \[  \forall t \geqslant 0, \forall x \in \mathbb{D}, \quad \mathbb{E} \big(W_{t +
     \theta}^{\tau, t, x} - W_t^{\tau, t, x}\big) = 0, \]
  \[ \forall t \geqslant 0, \forall x \in (1 - \varepsilon) \mathbb{D}, \quad
     \mathbb{E} \big((W_{t + \theta}^{\tau, t, x} - W_t^{\tau, t, x})^{\alpha}\big)
     =\tmmathbf{1} (\alpha_1 \neq \alpha_2) \theta \big[1 + c_T (\varepsilon)\big],
  \]
  \[ \forall t \geqslant 0, \forall x \in \mathbb{D}, \forall p \geqslant 2,
     \quad \mathbb{E} \big(| W_{t + \theta}^{\tau, t, x} - W_t^{\tau, t, x} |^p\big)
     \lesssim \theta^{p / 2} . \]
\end{lemma}

\begin{proof*}{Proof of Lemma \ref{L: discrete moments}}
  {\dueto{Discrete Moments}}For simplicity, we omit the superfix
  $\tau_{\varepsilon}$ in the discrete stopped processes
  $X^{\tau_{\varepsilon}}$ and $B^{\tau_{\varepsilon}}$. Recall that
  $(\mathcal{F}_k)_{k \in [0, N^5]}$ is the filtration associated to the
  discrete random walk $(B_k)_{k \in [0, N^5]}$, and
  $(\widetilde{\mathcal{F}}_n)_{n \in [0, N^4]}$, with
  $\widetilde{\mathcal{F}}_n \assign \mathcal{F}_{n N}$, the filtration
  associated to the coarse random walk $(X_n)_{n \in [0, N^4]}$. Recall
  finally that $X_n \assign B_{n N}$, and therefore $\mathd X_{n + 1} \assign
  X_{n + 1} - X_n = \sum_{l = 1}^N \mathd B_{n N + l}$. We are only interested
  in increments occuring before stopping, otherwise the estimate is trivial,
  all increments being zero after stopping. We have for all $n \leqslant N^4$, that 
   $\mathbb{E} (\mathd X_{n + 1} | \widetilde{\mathcal{F}}_n) = \sum_{l =
     1}^N \mathbb{E} (\mathd B_{n N + l} | \mathcal{F}_{n N}) = 0$
  since $B$ is a martingale. Now for the second order moments, let $\alpha =
  (\alpha_1, \alpha_2)$ a multiindex with $| \alpha | = 2$. We estimate for
  example the variance of the first coordinate
  \[ 
    \mathbb{E} ((\mathd X^1_{n + 1})^2 | \widetilde{\mathcal{F}}_n)  \assign
     \mathbb{E} \big( \sum_{l, l' = 1}^N \mathd B_{n N + l}^1 \mathd B_{n N
    + l'}^1 | \mathcal{F}_{n N} \big) = \sum_{l = 1}^N \mathbb{E} ((\mathd
    B_{n N + l}^1)^2 | \mathcal{F}_{n N})
  \]
  where we used that $B$ is a martingale. We have in the sum above for $l =
  1$,
  \[ 
    \mathbb{E} ((\mathd B_{n N + 1}^1)^2 | \mathcal{F}_{n N})  =  \mathbb{E}
    \big( \tmmathbf{1} (\varepsilon_{n N} = + 1)^2 \varepsilon_{n N + 1}^2 
     \sqrt{2 \delta} ^2 | \mathcal{F}_{n N} \big) =\tmmathbf{1}
    (\varepsilon_{n N} = + 1) 2 \delta .
  \]
  For the other summands, with $l \geqslant 2$, we have
  \begin{align*}\MoveEqLeft
    \mathbb{E} \big((\mathd B_{n N + l}^1)^2 | \mathcal{F}_{n N}\big)  
    =  \mathbb{E}
    \big( \tmmathbf{1} (\varepsilon_{n N + l - 1} = + 1)^2 \varepsilon_{n N +
    l}^2   \sqrt{2 \delta} ^2 | \mathcal{F}_{n N} \big)\\
     = & 2 \delta \; \mathbb{E} \big(\tmmathbf{1} (\varepsilon_{n N + l - 1}
    = + 1)  | \mathcal{F}_{n N}\big) = 2 \delta \Big[ \frac{1}{2} \cdot 0 +
    \frac{1}{2} \cdot 1 \Big] = \delta .
  \end{align*}
  Since $\theta = N \delta$, we have, for $\alpha = (2, 0)$,
  \[ 
    \mathbb{E} \big((\mathd X^1_{n + 1})^2 | \widetilde{\mathcal{F}}_n\big)  = 
    \big[\tmmathbf{1} (\varepsilon_{n N} = + 1) 2 \delta + (N - 1) \delta\big] =
    \theta (1 + c (\varepsilon)),
  \]
  and the same estimate holds for the second coordinate. Finally since $\mathd
  B^1_n (x) \mathd B^2 (x) = 0$ for all $n$, all $x$, it follows $\mathbb{E}
  (\mathd X^1_{n + 1} \mathd X^2_{n + 1} | \widetilde{\mathcal{F}}_n) = 0$ for
  all $n$.
  
  For the higher order moments, if $| \alpha | = p$, $\mathbb{E} (\mathd X_{n
  + 1}^{\alpha} | \widetilde{\mathcal{F}}_n) \lesssim \mathbb{E} (| \mathd
  X_{n + 1}^1 |^p | \widetilde{\mathcal{F}}_n) +\mathbb{E} (| \mathd X_{n +
  1}^2 |^p | \widetilde{\mathcal{F}}_n)$. For integer moments, estimate for
  example
  \begin{align*} \MoveEqLeft
    \mathbb{E} \big(| \mathd X_{n + 1}^1 |^{2 p} \, | \widetilde{\mathcal{F}}_n\big)  =
    \sum_{l_1, \ldots, l_{2 p}} \mathbb{E} \big(\mathd B^1_{n N + l_1} \ldots
    \mathd B_{n N + l_{2 p}}^1 \, | \widetilde{\mathcal{F}}_n\big) \\
     \leqslant & \frac{(2 p) !}{2} \sum_{l_1, l_3, \ldots l_{2 p - 1}}
    \mathbb{E} \big((\mathd B^1_{n N + l_1})^2 (\mathd B^1_{n N + l_3})^2 \ldots
    (\mathd B_{n N + l_{2 p - 1}}^1)^2 \, | \widetilde{\mathcal{F}}_n\big)
    \\
     \lesssim & N^p \delta^p \lesssim \theta^p . \nonumber
  \end{align*}
  
  By H{\"o}lder, we get for any $p \geqslant 2$, $\mathbb{E} (| \mathd X_{n +
  1}^1 |^p \;| \widetilde{\mathcal{F}}_n) \lesssim \theta^{p / 2}$.
\end{proof*}

\begin{proof*}{Proof of Lemma \ref{L: continuous moments}}
  {\dueto{Continuous Moments}}The first statement is obvious. For the second
  statement, notice that $x \in (1 - \varepsilon) \mathbb{D}$ implies
  $\tmop{dist} (x, \partial \mathbb{D}) \geqslant \varepsilon$. We now take
  advantage of the fact that the standard mean deviation of the non--stopped
  Brownian motion $W_{t + \theta}^{t, x} - W_t^{t, x}$ over one time step
  $\theta$  is $\sigma_{\theta} = \sqrt{\theta} = \sqrt{T / N^4} = \varepsilon
  \sqrt{T} / N$. Therefore $\sigma_{\theta} \ll \varepsilon$, i.e.
  $\sigma_{\theta} = c_T (\varepsilon) \varepsilon .$ This means that for $x
  \in (1 - \varepsilon) \mathbb{D}$, we have $\tmop{dist} (x, \partial
  \mathbb{D}) \geqslant \varepsilon$.In other words the point $x$ is ``very
  far away'' from the boundary as compared to the distance the Brownian motion
  can diffuse. More precisely, as a consequence of standard estimates of first
  hitting times of the Brownian motion, we have for such an $x$ and the
  corresponding stopped process $W_t^{\tau, t, x}$, that
  $\mathbb{P} (t + \theta > \tau) = c (\sigma_{\theta} / \varepsilon) = c
     (c_T (\varepsilon)) = c_T (\varepsilon) .$ 
  Note now $(W^{\tau, t, x, i}_s)_{i = 1, 2}$ the two components of the
  Brownian motion. Consider for example the second moment of the first
  coordinate. Using It{\^o} formula and taking expectation yields
   \begin{align*}\MoveEqLeft
    \mathbb{E} \big(W_{t + \theta}^{\tau, t, x, 1} - W_t^{\tau, t, x, 1}\big)^2 
    = \frac{1}{2} \mathbb{E} \int_t^{(t + \theta) \wedge \tau} \mathd [W^{\tau,
    t, x, 1}, W^{\tau, t, x, 1}]_s \\
     =& \frac{1}{2} \mathbb{E} \Big( \int_t^{(t + \theta)} \mathd [W^{\tau,
    t, x, 1}, W^{\tau, t, x, 1}]_s  \; | t + \theta \leqslant \tau \Big)
    \mathbb{P} (t + \theta \leqslant \tau) \nonumber\\
    &  \quad  + \frac{1}{2} \mathbb{E} \Big( \int_t^{(t + \theta)} \mathd [W^{\tau,
    t, x, 1}, W^{\tau, t, x, 1}]_s \; | t + \theta > \tau \Big) \mathbb{P} (t
    + \theta > \tau) \nonumber\\
     \backassign & \theta \delta_{\alpha_1 \alpha_2} \mathbb{P} (t + \theta
    \leqslant \tau) + A (t, x, \theta) \mathbb{P} (t + \theta > \tau),
    \nonumber
  \end{align*}
  
  where clearly $0 \leqslant A (t, x, \theta) \leqslant \theta$ for all $(t,
  x)$. This yields the second statement since $\mathbb{P} (t + \theta > \tau)
  = c_T (\varepsilon)$ and $\mathbb{P} (t + \theta \leqslant \tau) = 1 - c_T
  (\varepsilon)$ in the case $\alpha = (2, 0)$. The case $\alpha = (0, 2)$ is
  similar. The case $\alpha = (1, 1)$ is trivial since $\mathd [W^{\tau, t, x,
  1}, W^{\tau, t, x, 2}]_s = 0$. The third statement of the Lemma follows the
  same lines using the known higher moments for the Brownian motion.
\end{proof*}

\subsection{Auxiliary convergence results}\label{S: auxiliary convergence
results}

The goal of this section is to prove the following two convergence results:

\begin{lemma}[Weak convergence]
  \label{L: weak convergence}  Let $T > 0$. Let $\psi$
  harmonic on $\mathbb{D}$ with $\psi$ smooth on $\partial \mathbb{D}$. Assume
  weak consistency. Then we have
  \[ \mathbb{E} \psi (X_T^{\tau_{\varepsilon}}) =\mathbb{E} \psi (W_T^{\tau})
     + c_{\psi, T} (\varepsilon) + c (T). \]
\end{lemma}

\begin{lemma}[Weak convergence of
  discrete martingale transforms]
  \label{L: convergence discrete martingale transforms} Let $T > 0$. Let $f$ as above. Then
  \[ \| f (X_T^{\tau_{\varepsilon}}) - M_T^f \|_p = c_T (\varepsilon). \]
\end{lemma}

\begin{proof*}{Proof of Lemma \ref{L: weak convergence}}
  Let $\psi$ harmonic on $\mathbb{D}$, with $\psi$ smooth on $\partial
  \mathbb{D}$. We first split
  \[ \mathbb{E} \psi (X_T^{\tau_{\varepsilon}}) =\mathbb{E} (\psi
     (X_T^{\tau_{\varepsilon}}) | \tau_{\varepsilon} \leqslant T) \mathbb{P}
     (\tau_{\varepsilon} \leqslant T) +\mathbb{E} (\psi
     (X_T^{\tau_{\varepsilon}}) | \tau_{\varepsilon} > T) \mathbb{P}
     (\tau_{\varepsilon} > T). \]
  We claim that the second term is small uniformly w.r.t. $\varepsilon$ when
  $T$ is large. Indeed, by definition of $\tau_{\varepsilon}$, this term
  collects the contribution of those trajectories that remained in the disc
  $(1 - \varepsilon) \mathbb{D}$ during the whole interval of time $[0, T]$.
  We claim that this is small for $T$ large. Indeed, let $\tilde{B} =
  (\tilde{B}^1, \tilde{B}^2)$ the rotation of angle $\pi / 4$ of $B$, i.e.
  $\tilde{B}^1 \assign (B^1 + B^2) / \sqrt{2}$, and $\tilde{B}^2 \assign (B^1
  - B^2) / \sqrt{2}$. It follows
  \[ \tilde{B}^1_k (x) = \sum_{l = 1}^k \varepsilon_l (x)  \sqrt{\delta},
     \quad \tilde{B}^2_k (x) = \sum_{l = 1}^k \varepsilon_{l - 1} (x)
     \varepsilon_l (x)  \sqrt{\delta}, \]
  that is both $\tilde{B}^1$ and $\tilde{B}^2$ are (non independent) standard
  centered discrete random walks. Let 
  $\tilde{\tau}^1 \assign \inf \{ t_k ; | \tilde{B}_{t_k}^1 | \geqslant 1
     \}, \quad \tilde{\tau}^2 \assign \inf \{ t_k ; | \tilde{B}_{t_k}^2 |
     \geqslant 1 \}$  the first exit times.
  But
  $  
    \tau_{\varepsilon} > T  \Leftrightarrow  X_n \in (1 - \varepsilon)
    \mathbb{D}, \; n \in [0, N^4]
     \Rightarrow  B_{n N} \in (1 - \varepsilon) \mathbb{D}, \; n \in [0,
    N^4]
     \Rightarrow  B_k \in \mathbb{D}, \quad k \in [0, N^5]
     \Rightarrow  \tilde{\tau}^1 > T \infixand \tilde{\tau}^2 > T,
  $  
  where we have used that $| B_k - B_{n N} | \leqslant N \sqrt{\delta} \ll
  \varepsilon$ for $k \in [(n - 1) N + 1, \ldots, n N]$. In particular,
  $  \mathbb{P} (\tau_{\varepsilon} > T) \leqslant \mathbb{P} (\tilde{\tau}^1
     > T) = c (T) . $ 
  The last equality is a consequence of first hitting time estimates of
  standard centered discrete random walks, see e.g. {\tmname{Lawler}}
  {\cite{Law2010a}}. The function $\psi$ being bounded on $\mathbb{D}$, we
  have also $\mathbb{E} (\psi (X_T^{\tau_{\varepsilon}}) | \tau_{\varepsilon}
  > T) \mathbb{P} (\tau_{\varepsilon} > T) = c (T)$. Similarly for the second
  term, 
  \[\mathbb{E} \big(\psi (X_T^{\tau_{\varepsilon}})\; | \tau_{\varepsilon}
  \leqslant T\big) \mathbb{P} (\tau_{\varepsilon} \leqslant T) =\mathbb{E} \big(\psi
  (X_T^{\tau_{\varepsilon}}) \;| \tau_{\varepsilon} \leqslant T\big) + c_{\psi}
  (T).\] 
  On the other hand, since $\psi$ is harmonic, we have immediately
  $\mathbb{E} \psi (W_T^{\tau}) = \psi (W_0^{\tau}) = \psi (0, 0),$ 
  so that
  \begin{align*}\MoveEqLeft
   \mathbb{E} \psi (X_T^{\tau_{\varepsilon}}) -\mathbb{E} \psi (W_T^{\tau})
     =\mathbb{E} \big(\psi (X_T^{\tau_{\varepsilon}}) \; | \tau_{\varepsilon}
     \leqslant T\big) - \psi (0, 0) + c (T) \\
     =&\mathbb{E} \big(\psi
     (X_T^{\tau_{\varepsilon}}) - \psi (X_0^{\tau_{\varepsilon}})\; |
     \tau_{\varepsilon} \leqslant T\big) + c_{\psi} (T). 
       \end{align*}
  Now since $\mathbb{E} \psi (W_{t_{n + 1}}^{\tau, t_n, x}) = \psi
  (W_{t_n}^{\tau, t_n, x}) = \psi (x)$ for all $x \in \mathbb{D}$, we have
  \begin{align*}\MoveEqLeft
     \mathbb{E} \big(\psi (X_T^{\tau_{\varepsilon}}) - \psi
    (X_0^{\tau_{\varepsilon}})\big) 
    =\mathbb{E} \sum_{n = 1}^{n_{\varepsilon}}
    \big[\psi (X_{t_n}^{\tau_{\varepsilon}}) - \psi (X_{t_{n -
    1}}^{\tau_{\varepsilon}})\big] \\
     =& \mathbb{E} \sum_{n =1}^{n_{\varepsilon}} \Big[
     \psi(X_{t_n}^{\tau_{\varepsilon}}) - \psi (X_{t_{n -1}}^{\tau_{\varepsilon}}) - \big\{ \psi ( W_{t_n}^{\tau, t_{n - 1},
    X_{t_{n - 1}}^{\tau_{\varepsilon}}} ) - \psi (X_{t_{n -
    1}}^{\tau_{\varepsilon}})\big\}\Big] \\
     =&\mathbb{E} \sum_{n = 1}^{n_{\varepsilon}} \Big[ D \psi (X_{t_{n -
    1}}^{\tau_{\varepsilon}}) \cdot \big\{ (X_{t_n}^{\tau_{\varepsilon}} -
    X_{t_{n - 1}}^{\tau_{\varepsilon}}) - \big( W_{t_n}^{\tau, t_{n - 1},
    X_{t_{n - 1}}^{\tau_{\varepsilon}}} - X_{t_{n - 1}}^{\tau_{\varepsilon}}
    \big) \big\}  \nonumber\\
    &\quad  + \frac{1}{2} \sum_{| \alpha | = 2} D^{\alpha} \psi (X_{t_{n -
    1}}^{\tau_{\varepsilon}}) \cdot \big\{ (X_{t_n}^{\tau_{\varepsilon}} -
    X_{t_{n - 1}}^{\tau_{\varepsilon}})^{\alpha} - \big( W_{t_n}^{\tau, t_{n
    - 1}, X_{t_{n - 1}}^{\tau_{\varepsilon}}} - X_{t_{n -
    1}}^{\tau_{\varepsilon}} \big)^{\alpha} \big\} \nonumber\\
    &\quad  +  R (X_{t_{n - 1}}^{\tau_{\varepsilon}},
    X_{t_n}^{\tau_{\varepsilon}}) - R \big( X_{t_{n -
    1}}^{\tau_{\varepsilon}}, W_{t_n}^{\tau, t_{n - 1}, X_{t_{n -
    1}}^{\tau_{\varepsilon}}} \big) \Big] \nonumber \backassign A + B + C, \nonumber
  \end{align*}
  where $R (x, y)$ is the Taylor rest
  $R (x, y) \assign \frac{1}{3!} \sum_{| \alpha | = 3} D^{\alpha} \psi (x +
     \theta_{x, y} (y - x)) \cdot (y - x)^{\alpha}, \quad x, y \in \mathbb{D},
     \theta_{x, y} \in [0, 1]$.
  Using the weak consistency of $X^{\tau_{\varepsilon}}$ with $W^{\tau}$, we
  get, recalling $T = N^4 \theta$,
  \begin{align*}\MoveEqLeft
    | A | 
     \leqslant  \mathbb{E} \sum_{n = 1}^{n_{\varepsilon}} \Big[ | D
    \psi (X_{t_{n - 1}}^{\tau_{\varepsilon}}) | \cdot \big| \mathbb{E}
    (X_{t_n}^{\tau_{\varepsilon}} - X_{t_{n - 1}}^{\tau_{\varepsilon}} |
    \mathcal{F}_{n - 1} \nobracket) -\mathbb{E} \big( W_{t_n}^{\tau, t_{n -
    1}, X_{t_{n - 1}}^{\tau_{\varepsilon}}} - X_{t_{n -
    1}}^{\tau_{\varepsilon}} | \mathcal{F}_{n - 1} \nobracket \big) \big|
    \Big]\\
     \leqslant & \| D \psi \|_{\infty}  \mathbb{E} \sum_{n =
    1}^{n_{\varepsilon}} \theta c (\varepsilon) 
      \leqslant \| D \psi
    \|_{\infty} N^4 \theta c (\varepsilon) 
    \lesssim c_{\psi} (\varepsilon) T 
    = c_{\psi, T} (\varepsilon),
  \end{align*}
  \begin{align}\MoveEqLeft
    | B |  \leqslant \mathbb{E} \sum_{n = 1}^{n_{\varepsilon}} \sum_{| \alpha
    | = 2} | D^{\alpha} \psi (X_{t_{n - 1}}^{\tau_{\varepsilon}}) | \times
    \nonumber\\
    & \qquad \big| \mathbb{E} \big((X_{t_n}^{\tau_{\varepsilon}} - X_{t_{n -
    1}}^{\tau_{\varepsilon}})^{\alpha} | \mathcal{F}_{n - 1} \big)
    -\mathbb{E} \big( ( W_{t_n}^{\tau, X_{t_{n -
    1}}^{\tau_{\varepsilon}}, t_n} - X_{t_{n - 1}}^{\tau_{\varepsilon}}
    )^{\alpha} | \mathcal{F}_{n - 1} \nobracket \big) \big|
    \nonumber\\
    & \lesssim \| D^2 \psi \|_{\infty} \quad \mathbb{E} \sum_{n =
    1}^{n_{\varepsilon}} \theta c (\varepsilon) \lesssim c_{\psi}
    (\varepsilon) T = c_{\psi, T} (\varepsilon), \nonumber
  \end{align}
  recalling that for the second moments, $X_{t_{n - 1}}^{\tau_{\varepsilon}} \in (1 - \varepsilon) \mathbb{D}$.
  Finally, since third order moments are at most of order $\theta^{3 / 2}$, we
  deduce
  \begin{align*}\MoveEqLeft 
  | C | =  \Big| \mathbb{E} \sum_{n = 1}^{n_{\varepsilon}} 
  \mathbb{E} 
  \big(R(X_{t_{n - 1}}^{\tau_{\varepsilon}}, X_{t_n}^{\tau_{\varepsilon}}) |
    \mathcal{F}_{n - 1} \big) 
    +\mathbb{E} \big( R ( X_{t_{n -
    1}}^{\tau_{\varepsilon}}, W_{t_n}^{\tau, X_{t_{n -
    1}}^{\tau_{\varepsilon}}, t_n} ) | \mathcal{F}_{n - 1} \nobracket
    \big) \Big|\\
     \lesssim & \| D^3 \psi \|_{\infty} \Big\{ \mathbb{E} \sum_{n =
    1}^{n_{\varepsilon}} 
    \mathbb{E} \big(| X_{t_n}^{\tau_{\varepsilon}} - X_{t_{n
    - 1}}^{\tau_{\varepsilon}} |^3 | \mathcal{F}_{n - 1} \big)
    +\mathbb{E} \big( | W_{t_n}^{\tau, X_{t_{n -
    1}}^{\tau_{\varepsilon}}, t_n} - X_{t_{n - 1}}^{\tau_{\varepsilon}}
    |^3 | \mathcal{F}_{n - 1} \nobracket \big)\Big\}\\
     \lesssim & \| D^3 \psi \|_{\infty} \mathbb{E} \sum_{n =
    1}^{n_{\varepsilon}} \theta^{3 / 2} \lesssim \| D^3 \psi \|_{\infty} T
    \theta^{1 / 2} = c_{\psi, T} (\varepsilon) .
  \end{align*}
  This concludes the proof of the weak convergence.
\end{proof*}

\begin{proof*}{Proof of Lemma \ref{L: convergence discrete martingale
transforms}}{\dueto{Convergence of discrete martingale transforms}}
  We aim at estimating $\| f (X_T^{\tau_{\varepsilon}}) - M_T^f \|_p \assign
  (\mathbb{E} | f (X_T^{\tau_{\varepsilon}}) - M_T^f |^p)^{1 / p}$. Split
  first
  \begin{align*}\MoveEqLeft
    f (X_T^{\tau_{\varepsilon}}) - f (X_0^{\tau_{\varepsilon}})
     =  f(B_T^{\tau_{\varepsilon}}) - f (B_0^{\tau_{\varepsilon}})
    = \sum_{k =
    1}^{k_{\varepsilon}} [f (B_k^{\tau_{\varepsilon}}) - f (B_{k -
    1}^{\tau_{\varepsilon}})]\\
     = & \sum_{k = 1}^{k_{\varepsilon}} \sum_{i = 1, 2} \partial_i f (B_{k -
    1}) \mathd B_k^i + \frac{1}{2} \sum_{k = 1}^{k_{\varepsilon}} \sum_{i, j =
    1, 2} \partial^2_{i j} f (B_{k - 1}) \mathd B^i_k \mathd B^j_k
       + \sum_{k = 1}^{k_{\varepsilon}} R_3^f (B_{k - 1}, \mathd B_k),
  \end{align*}
  and on the other hand we have
    $M_T^f - M_0^f  \assign  \sum_{k = 1}^{k_{\varepsilon}} \sum_{i = 1, 2}
    \partial_i f (B_{k - 1}) \mathd B_k^i .$
  Since $f (X_0^{\tau_{\varepsilon}}) = M_0^f = f (0)$, it follows simply
  \[ 
    f (X_T^{\tau_{\varepsilon}}) - M_T^f  =  \frac{1}{2} \sum_{k =
    1}^{k_{\varepsilon}} \sum_{i, j = 1, 2} \partial^2_{i j} f (B_{k - 1})
    \mathd B^i_k \mathd B^j_k + \sum_{k = 1}^{k_{\varepsilon}} R_3^f (B_{k -
    1}, \mathd B_k)
    \backassign 
    A + B.
  \]
  For the second term above, we observe that for all $k$, all $x$,
  $ | R_3^f (B_{k - 1} (x), \mathd B_k (x)) | \lesssim \| D^3 f \|_{\infty}
     \delta^{3 / 2}, $ 
  and therefore
  \[ \| B \|_p \lesssim \sum_{k = 1}^{N^5} \| D^3 f \|_{\infty} \delta^{3 / 2}
     = \| D^3 f \|_{\infty} (N^5 \delta) \delta^{1 / 2} = c_T (\varepsilon) .
  \]
  Recalling that $\tau_{\varepsilon} = n_{\varepsilon} \delta =
  k_{\varepsilon} \theta$ or equivalently $k_{\varepsilon} = N
  n_{\varepsilon}$, we split the sum $A$ into blocks of size $N$, namely $A =
  \sum_{n = 1}^{n_{\varepsilon}} A_n$, with
  \begin{align*}\MoveEqLeft
    A_n =  \frac{1}{2} \sum_{l = 1}^N \Big[\partial^2_{11} f (B_{(n - 1) N + l -
    1})  (\mathd B_{(n - 1) N + l}^1)^2 + \partial^2_{22} f (B_{(n - 1) N + l
    - 1})  (\mathd B_{(n - 1) N + l}^2)^2\Big]\\
     = & \frac{\delta}{2} \sum_{l = 1}^N \Big[\partial^2_{11} f (B_{(n - 1) N + l
    - 1}) + \partial^2_{22} f (B_{(n - 1) N + l - 1})\Big]\\
      & + \frac{\delta}{2} \sum_{l = 1}^N \Big[\partial^2_{22} f (B_{(n - 1) N +
    l - 1}) - \partial^2_{11} f (B_{(n - 1) N + l - 1})\Big] \varepsilon_{(n - 1)
    N + l - 1}\\
     = & \delta \sum_{l = 1}^N \partial^2_{22} f (B_{(n - 1) N + l - 1})
    \varepsilon_{(n - 1) N + l - 1},
  \end{align*}
  where we used $(\mathd B_{k}^i)^2 = \frac{\delta}{2}\big( 1 + (-1)^i \varepsilon_{k} \big)$ followed by the harmonicity of $f$. We split further
  \begin{align}
  \MoveEqLeft
    A_n  = \delta \sum_{l = 1}^N \Big[\partial^2_{22} f (B_{(n - 1) N + l - 1}) -
    \partial^2_{22} f (B_{(n - 1) N})\Big] \varepsilon_{(n - 1) N + l - 1}
    \nonumber\\
    & \qquad + \delta \partial^2_{22} f (B_{(n - 1) N}) \sum_{l = 1}^N
    \varepsilon_{(n - 1) N + l - 1} \nonumber \backassign B_n + C_n \nonumber
  \end{align}
  For $B_n$, we observe $| \partial^2_{22} f (B_{(n - 1) N + l - 1}) -
  \partial^2_{22} f (B_{(n - 1) N}) | \lesssim \| D^3 f \|_{\infty} l
  \sqrt{\delta}$, therefore $\| B_n \|_p \lesssim N^2 \delta^{3 / 2}$ and
  \[ \big\| \sum_{n = 1}^{N^4} B_n \big\|_p \lesssim N^6 \delta^{3 / 2} =
     (N^5 \delta) N (T N^{- 5})^{1 / 2} = c_{\varepsilon} (T), \]
  \[ \| C_n \|_p \lesssim \delta \| D^2 f \|_{\infty}  \Big( \mathbb{E}
     \Big| \sum_{l = 1}^N \varepsilon_{(n - 1) N + l - 1} \Big|^p
     \Big)^{1 / p} . \]
  Notice that the sum above is $\sum_{l = 1}^N \varepsilon_{(n - 1) N + l - 1}
  = \mathd X_n^1 + \mathd X^2_n$, and we know from the moment estimates that
  $\| \mathd X_n^i \|_p \lesssim \theta^{1 / 2}$, $i = 1, 2$. We conclude
  \[ \big\| \sum_{n = 1}^{N^4} C_n \big\|_p \lesssim N^4 \delta \| D^2 f
     \|_{\infty} (N \delta)^{1 / 2} = (N^5 \delta) \| D^2 f \|_{\infty} N^{- 1
     / 2} \delta^{1 / 2} = c_T (\varepsilon) .\]
  This concludes the proof of Lemma \ref{L: convergence discrete martingale
  transforms}.
\end{proof*}

\subsection{Finalizing the proof of the upper bound}\label{S: main results}
We start with the proof of the convergence theorem:

\begin{proof*}{Proof of Theorem \ref{T: convergence}}
  {\dueto{Convergence of $L^p$ norms of martingales}}Recall that we want to
  prove
   $\lim_{T \rightarrow \infty} \lim_{\varepsilon \rightarrow 0} \mathbb{E} |
     M_T^f |^p =\mathbb{E} | f (W^{\tau}_{\infty}) |^p.$ 
  
  Split first as the sum of three differences
   \begin{align*} 
   \MoveEqLeft
       \big\| f (W^{\tau}_{\infty}) \big\|_p - \big\| M_T^f \big\|_p 
       = \big\| f (W^{\tau}_{\infty}) \big\|_p - \big\| f (W^{\tau}_T) \big\|_p + \big\| f
       (W^{\tau}_T) \big\|_p \\
       & - \big\| f (X^{\tau_{\varepsilon}}_T) \big\|_p + \big\| f
       (X^{\tau_{\varepsilon}}_T) \big\|_p - \big\| M_T^f \big\|_p
        \backassign  A + B + C.
     \end{align*} 
  As seen before, we have $\mathbb{P} (\tau > T) = c (T)$, therefore $| A | =
  c (T)$. For the third term, we have simply $| C | \leqslant \| f
  (W^{\tau}_T) - f (X^{\tau_{\varepsilon}}_T) \|_p = c_T (\varepsilon)$ thanks
  to Lemma \ref{L: convergence discrete martingale transforms}.
  
  For the second term, define successively on the boundary $\partial
  \mathbb{D}$, $\psi (x) \assign | f (x) |_X^p$ and $\psi_{\eta} \assign \psi
  \ast \rho_{\eta}$ a mollified version of $\psi$ tending to $\psi$ when
  $\eta$ goes to zero. We also denote $\psi$ (resp. $\psi_{\eta}$) defined on
  $\mathbb{D}$ the Poisson extension of $\psi_{| \partial \mathbb{D}}$ (resp.
  $\psi_{| \partial \mathbb{D}}$). Since $f \in L^{\infty} (\partial
  \mathbb{D}; X)$, it follows that $\psi$ and $\psi_{\eta}$ are bounded in
  $\mathbb{D}$, and $\psi_{\eta} = \psi + c (\eta)$ in $L^{\infty}
  (\mathbb{D})$. Finally, notice that if $| x | < 1$, then we have for the
  Poisson extensions $\psi (x) \neq | f (x) |_X^p$. However $\psi (x) = | f
  (x) |_X^p + c (\varepsilon)$ for those $x$'s next to the boundary, i.e. $1 -
  \varepsilon < | x | \leqslant 1$. We can now estimate the first term of $B$:
  \begin{align*}\MoveEqLeft
    \big\| f (W_T^{\tau}) \big\|_p  =  \mathbb{E} (| f (W_T^{\tau}) |^p)^{1 / p}
    =\mathbb{E} \big(| f (W_T^{\tau}) |^p \, | \, T > \tau \big)^{1 / p} + c (T)\\
     = & \mathbb{E} \big(\psi (W_T^{\tau}) \, | \, T > \tau \big)^{1 / p} + c (T)
    =\mathbb{E} (\psi (W_T^{\tau}))^{1 / p} + c (T)\\
     = & \mathbb{E} (\psi_{\eta} (W_T^{\tau}))^{1 / p} + c (\eta) + c (T),
  \end{align*}
  where we have used that $W_T^{\tau} \in \partial \mathbb{D}$ when $T >
  \tau$. Similarly, since $\mathbb{P} (\tau_{\varepsilon} > T) = c (T)$ and $1
  - \varepsilon \leqslant X_T^{\tau_{\varepsilon}} \leqslant 1$ for $T >
  \tau_{\varepsilon}$, we have
  \begin{align*}\MoveEqLeft
    \big\| f (X_T^{\tau_{\varepsilon}}) \big\|_p  =  \mathbb{E} 
    \big(| f(X_T^{\tau_{\varepsilon}}\big) |^p\big)^{1 / p} 
    =\mathbb{E} \big(|f(X_T^{\tau_{\varepsilon}}) |^p \,| \, T > \tau_{\varepsilon}\big)^{1 / p} + c (T)\\
     = & \mathbb{E} \big(\psi (X_T^{\varepsilon})\, |\, T > \tau_{\varepsilon})^{1 /
    p} + c (\varepsilon) + c (T)
     =  \mathbb{E} (\psi (X_T^{\varepsilon})\big)^{1 / p} + c (\varepsilon) + c
    (T)\\
     = & \mathbb{E} (\psi_{\eta} (X_T^{\varepsilon}))^{1 / p} + c (\eta) + c
    (\varepsilon) + c (T).
  \end{align*}
  It follows,
  \begin{align*}\MoveEqLeft
    | B |  \leqslant  \big| \mathbb{E} (\psi_{\eta} (W_T^{\tau}))^{1 / p}
    -\mathbb{E} (\psi_{\eta} (X_T^{\varepsilon}))^{1 / p} \big| + c (\eta) + c
    (\varepsilon) + c (T)\\
     = & c_{\eta, T} (\varepsilon) + c (\eta) + c (\varepsilon) + c (T),
  \end{align*}
  where we used Lemma \ref{L: weak convergence} for the second line. Finally
  \[ 
    \big\| f (W^{\tau}_{\infty}) \big\|_p - \big\| M_T^f \big\|_p  =  A + B + C = c_{\eta,
    T} (\varepsilon) + c (\eta) + c (\varepsilon) + c (T) .
  \]
  Fix any small $\eta > 0$, choose $T > 0$ large enough so that $c (T)
  \leqslant \eta$, then $\varepsilon > 0$ small enough so that $c_{\eta, T}
  (\varepsilon) + c (\varepsilon) \leqslant \eta$. Hence
  \[ \lim_{T \rightarrow \infty} \lim_{\varepsilon \rightarrow 0} 
  \big| \big\| f
     (W^{\tau}_{\infty}) \big\|_p - \big\| M_T^f \big\|_p \big| \leqslant c (\eta), \]
  therefore
  $\lim_{T \rightarrow \infty} \lim_{\varepsilon \rightarrow 0} \| M_T^f
     \|_p = \| f (W^{\tau}_{\infty}) \|_p $ 
  as desired. This concludes the proof of Theorem \ref{T: convergence}.
\end{proof*}

\ 

\begin{proof*}{Proof of Theorem \ref{theorem_upper}}
  This is now a direct consequence of Theorem \ref{T: convergence}. Let $f \in
  L^p (\partial \mathbb{D})$. Its $L^p$ norm is directly related to the
  stochastic $L^p$ norm
  \[ \big\| f (W_{\infty}^{\tau}) \big\|_p \assign \big(\mathbb{E} | f (W_{\infty}^{\tau})
     |_X^p\big)^{1 / p} = \Big( \int_{\partial \mathbb{D}} | f (z) |_X^p 
     \frac{\mathd z}{2 \pi} \Big)^{1 / p} = \frac{1}{(2 \pi)^{1 / p}}  \| f
     \|_{L^p (\partial \mathbb{D})}, \]
  and the same relation holds for the smooth function $g \assign \mathcal{H}
  f$. From Theorem \ref{T: convergence} we know
  \[ \big\| f (W_{\infty}^{\tau}) \big\|_p = \lim_{T \rightarrow \infty}
     \lim_{\varepsilon \rightarrow 0} \big\| M_T^f \big\|_p, \qquad \big\| g
     (W_{\infty}^{\tau}) \big\|_p = \lim_{T \rightarrow \infty} \lim_{\varepsilon
     \rightarrow 0} \big\| M_T^g \big\|_p, \]
  and from Lemma \ref{L: Lp estimate for Mng} we know that $\| M_T^g \|_p \leqslant \| \Sha \|_{L^p_X \rightarrow L^p_X}  \| M_T^f
     \|_p $  for all $T > 0$,
  $\varepsilon > 0$.
  It follows that for any $f \in L^p_X (\partial \mathbb{D})$,
  \[ \| g \|_{L^p_X (\partial \mathbb{D})} = \| \mathcal{H} f \|_{L^p_X (\partial
     \mathbb{D})} \leqslant \| \Sha \|_{L^p_X \rightarrow L^p_X} \| f \|_{L^p_X
     (\partial \mathbb{D})}, \]
  that is $\| \mathcal{H} \|_{L^p_X \rightarrow L^p_X} \leqslant \| \Sha \|_{L^p_X
  \rightarrow L^p_X}$.
\end{proof*}

\section{Proof of the Lower Bound (Theorem \ref{theorem_lower})}
In this section, we will prove Theorem \ref{theorem_lower}.

\paragraph{Representation using angles adapted to $\Sha$}
Our starting point is the representation of a function $f$ with left and right tosses as in section \ref{S: Tuning}
\[ 
	f	 =	 \langle f \rangle_{I_0} + \mathd_0 f\ \varepsilon_0
			 + \sum_{k = 1}^{\infty} \sum_{\sigma=\pm}
		\mathd_k f^\sigma(\varepsilon_0,\varepsilon^-_{1},\varepsilon^+_{1},\ldots,\varepsilon^-_{k-1},\varepsilon^+_{k-1})\ \varepsilon^\sigma_k.
\]
A direct connection between $\Sha$ and $\mathcal{H}$ will require that we work with
two random generators in parallel, namely
\[
	\varphi^-(\theta) = \tmop{sqsin(\theta)} \assign \tmop{sign}(\tmop{sin}(\theta)),
	\quad \varphi^+(\theta) = \tmop{sqcos(\theta)} \assign \tmop{sign}(\tmop{cos}(\theta)),
\]
both defined on the same probability space $(\mathbb{T},\frac{\mathd\theta}{2\pi})$.
For $k=0$, we replace $\varepsilon_0$ by $\varphi^+ (\theta_0)$.
For $k\geqslant 1$ we replace $\varepsilon^{\pm}_{k}$ by $\varphi^{\pm}(\theta_k)$. We can now define
\[
	F_n(\vec{\theta}_{n-1})
	= \langle f \rangle_{I_0}
		+ \mathd_0 F\ \varphi^+(\theta_0)
		+ \sum_{k=1}^{n-1} \big[\mathd_k F^-(\vec{\theta}_{k-1}) \ \varphi^-(\theta_k)
			+ \mathd_k F^+(\vec{\theta}_{k-1}) \ \varphi^+(\theta_k)\big]
\]
with $\mathd F_0 = \mathd f_0$, and for $k \geqslant 1$
\[ \mathd F_k^{\pm}  (\vec{\theta}_{k-1}) = \mathd f_k^{\pm} 
   (\varphi^+ (\theta_0), \varphi^- (\theta_1), \varphi^+ (\theta_1), \ldots,
   \varphi^- (\theta_{k-1}), \varphi^+ (\theta_{k-1})) . \]
This is a martingale sequence in the obvious filtered probability space and the probability distributions of $f$ and $F$ are the same.
We refer to Figure \ref{Fig: coding angles complex} for an illustration of the trigonometric tosses adapted to $\Sha$.

\begin{figure}[ht]
\centering
\includegraphics[width=8cm]{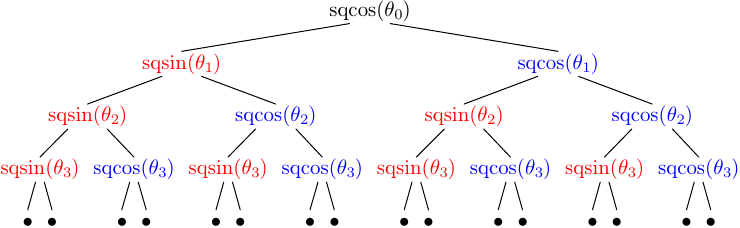}
\caption{Trigonometric tosses adapted to $\Sha$}
\label{Fig: coding angles complex}
\end{figure}

We will use expectation operators to denote integrals against probability measures.
For a function $f = f(x)$, with $x\in I_0=[0,1)$, and for a function $F=F(\vec{\theta})=F(\theta_0,\ldots,\theta_n)$, we note 
\[ \mathbb{E}^x f \assign \int_0^1 f(x) \mathd x, \quad 
\mathbb{E}^{\vec{\theta}} F\assign\int_0^{2\pi}\ldots\int_0^{2\pi}  F(\theta_0,\ldots,\theta_n)  \frac{\mathd\theta_0}{2\pi}\ldots \frac{\mathd\theta_n}{2\pi}.
\]
\begin{remark}\label{x-and-theta}
Using this notation, we have for the functions $F(\vec{\theta})$ and $f(x)$ related like above
$ 
\mathbb{E}^x f= \mathbb{E}^{\vec{\theta}} F.
$ 
\end{remark}

\paragraph{Comparing $\mathcal{H}$ and $\Sha$ in a projected form}

The operator $\Sha$ understood in the language of sign tosses maps as follows:
\[ \Sha : \mathd f^{\pm}_k  (\varepsilon_{0,} \varepsilon^+_1,
   \varepsilon^-_1, \ldots, \varepsilon^-_k) \varepsilon^{\pm}_{k + 1} \mapsto
   \pm \mathd f^{\pm}_k  (\varepsilon^+_1, \varepsilon^-_1, \ldots,
   \varepsilon^-_k) \varepsilon_{k + 1}^{\mp} . \]
and note that in the language involving angles we have $\Sha
\varphi^\pm = \pm \varphi^\mp$, that
is
$\Sha \varphi^{\sigma} = \sigma \varphi^{\bar{\sigma}}$
for $\sigma = \pm$ and $\bar{\sigma}$ the opposing sign.

Let us also define the projection $\mathcal{P}$ for periodic functions $f$ defined on $\mathbb{T}=[-\pi,\pi)$
by $(\mathcal{P} f) (x) = \sum^1_{i = - 2} \langle f \rangle_{A_i} \tmmathbf{1}_{A_i}(x)$, where the $4$ arcs
$A_i = [i \pi / 2, i \pi / 2 + \pi / 2)$ with $i\in\{ -1,0,1,2 \}$ form a partition of $\mathbb{T}$.

\begin{lemma}
  \label{lemma-projection-hilbert}There exists $c_0 > 0$ such that for
  signatures $\sigma=\pm 1$ there holds
  \[ \mathcal{P} \mathcal{H} \varphi^{\sigma} = c_0 \Sha \varphi^{\sigma} .
  \]
\end{lemma}

\begin{proof} 
	We start by identifying some elementary symmetries of $\varphi^\pm$, namely
$\varphi^\pm(x+\pi)=-\varphi^\pm(x) 
    $ and $ \varphi^\pm(x+\pi/2)=\mp\varphi^\mp(x).$  
We see that $\varphi^+$ is even and antisymmetric about $\pm\pi/2$ while $\varphi^-$ is odd and symmetric about $\pm\pi/2$. Let $\alpha(\theta)=\tmmathbf{1}_{[- \pi / 2, \pi / 2)}(\theta)$. Since $\alpha$ is even, $g=\mathcal{H}\alpha$ is odd. So $g$ has zeros at $0,\pm \pi$. One can check by inspecting the integral representation that for the singularities of $g$ there holds $\lim_{\theta\to \pm \pi/2} g(\theta)=\pm \infty$. We further have the odd function $\mathcal{H}\varphi^+$ and the even function $\mathcal{H}\varphi^-$. There hold
$ 
\mathcal{H}\varphi^+(\theta)=g(\theta)-g(\theta+\pi),\quad \mathcal{H}\varphi^-(\theta)=g(\theta-\pi/2)-g(\theta+\pi/2).
$ 
Thus we deduce that $\mathcal{H} \varphi^+ |_{(0, \pi)
  \setminus \{\pi / 2\}}$ is positive and symmetric about the axis $x = \pi /
  2$ and that $\mathcal{H} \varphi^+ |_{(- \pi, 0) \setminus \{- \pi / 2\}}$
  is negative and symmetric about $x = - \pi / 2$.
  Gathering the information,
  we obtain
  \[ \mathcal{P} \mathcal{H} \varphi^+ = c_0 \varphi^- = c_0 \Sha \varphi^+,
     \qquad \mathcal{P} \mathcal{H} \varphi^- = - c_0 \varphi^+ = c_0
     \Sha \varphi^- \]
  for some $c_0 > 0$. We refer to Section \ref{S: Constant_c0} for the calculation and value of the constant $c_0$.
  \end{proof}

\begin{remark} \label{flattening-hilbert}
Notice the elementary fact for functions $f, g \in L^2$ there holds
$  (\mathcal{P} f, g)_{L^2} = (\mathcal{P} f, \mathcal{P} g)_{L^2} = (f, \mathcal{P} g)_{L^2} . $ 
\end{remark}

\paragraph{$\mathcal{H}$ as a martingale transform. High frequency modulation}
We will not compare $\Sha$ with $\mathcal{H}$ directly, but rather $\Sha$
with a martingale transform version of $\mathcal{H}$. As observed by {\tmname{Bourgain}},
this is possible thanks to a succession of high frequency modulations. Given $F(\vec{\theta})=\lim_{n\to\infty} F(\vec{\theta}_n)$
we introduce $F^{\mathcal{H}}$,
\[
	F^{\mathcal{H}}(\vec{\theta})
	= \sum_{k=1}^{\infty}\sum_{\sigma=\pm} \mathd_k F^\sigma(\vec{\theta}_{k-1})
			\ (\mathcal{H}\varphi^\sigma)(\theta_k),
\]
where the first two terms of $F(\vec{\theta})$ are dropped and the Hilbert transform
acts on the last random generator of each summand.

We will require the following $L^p$ estimates. 

\begin{lemma}
  \label{lemma-inside-H}{\tmdummy}
  There holds
  \[
  	\big| \mathbb{E}^{\vec{\theta}}
		\big\langle
			F^{\mathcal{H}}(\vec{\theta}) , G(\vec{\theta})
		\big\rangle_{X, X^{\ast}}
	\big|
	\leqslant
	h_p \| F \|_{L_X^p} \| G \|_{L_{X^{\ast}}^q}
  \]

\end{lemma}
Notice that at difference with previous works, we are using a weak formulation for the $L^p$ estimate. This is a crucial feature of our analysis. It will allow  us later to make the link between the fundamental properties shared by $\mathcal{H}$ and $\Sha$ and the desired linear $L^p$ estimates.

\ 

\begin{proof*}{Proof of Lemma \ref{lemma-inside-H}}
Given $k\geqslant 1$, introduce the Fourier transform of the functions $\mathd_k F^\sigma(\vec{\theta}_{k-1})$,
with $\vec{\theta}_{k-1}=(\theta_0,\ldots,\theta_{k-1}) \in \mathbb{T}^k$ and $\varphi^\sigma(\theta_k)$ with $\theta_k\in\mathbb{T}$, namely
\[
	\mathd_k F^\sigma(\vec{\theta}_{k-1})
		= \sum_{\vec{l}_{k-1}\in\mathbb{Z}^k}
			y^\sigma_{\vec{l}_{k-1}} e^{\mathi\vec{l}_{k-1}\cdot\vec{\theta}_{k-1}},
	\quad
	\varphi^\sigma(\theta_k)
		= \sum_{l_k\in\mathbb{Z}\setminus \{0\} }
			c^\sigma_{l_k} e^{\mathi l_k \theta_k},
\]
where $\vec{l}_{k-1}=(l_0,\ldots,l_{k-1}) \in \mathbb{Z}^k$. The coefficients $y^\sigma_{\vec{l}_{k-1}}$ are Banach space valued, and the coefficients $c^\sigma_{l_k}$ are scalars.
We can exclude $l_k=0$ in the second sum since $\varphi^\sigma$, with $\sigma=\pm$, have average zero in $\mathbb{T}$.
By a standard approximation argument, we can further assume a finite spectrum for all terms of the form $\mathd_k F^\sigma(\vec{\theta}_{k-1})$. That is, there exists a sequence
of integers $\vec{N}=(N_k)_{k\geqslant 0}$ so that the above Fourier series are restricted to $\vec{l}_{k-1}\in\mathbb{Z}^k,\|\vec{l}_{k-1}\|_1\leqslant N_{k-1}$.

Introduce below a new variable $\psi\in\mathbb{T}$, facilitating the action of the Hilbert transform together with an increasing sequence of positive integers adapted to $\vec{N}$, governing the high frequency modulation. 
We claim

\begin{lemma}\label{moveH}
Given a sequence of spectral bounds $\vec{N}$ as above, there exists a sequence of high frequency modulations $\vec{n}=\vec{n}(\vec{N})$, such that for all $k\geqslant 1$, we have
\[ 
	\mathcal{H}_{\psi}  \big(\mathd F^{\sigma}_{k-1} (\vec{\theta}_{k-1} + \vec{n}_{k-1} \psi) \varphi^{\sigma} (\theta_{k} + n_{k} \psi)\big)
	=
	\mathd F^{\sigma}_k  (\vec{\theta}_{k-1} + \vec{n}_{k-1} \psi)  (\mathcal{H}\varphi^{\sigma})  (\theta_{k} + n_{k} \psi),
\]
where $\mathcal{H}_{\psi}$ denotes the Hilbert transform in the variable $\psi$.
\end{lemma}
\begin{proof}
For all $k\geqslant 1$, set 
$   
\mathcal{Z}_{k-1}\assign\{ \vec{l}_{k-1}\in\mathbb{Z}^{k}; \|\vec{l}_{k-1}\|_1\leqslant N_{k-1}\}
$  
so
\begin{align*}
\MoveEqLeft
 \mathcal{H}_{\psi}\big(\mathd F^{\sigma}_k  (\vec{\theta}_{k-1} + \vec{n}_{k-1} \psi)\varphi^{\sigma} (\theta_{k} + n_{k} \psi)\big)\\
=&\mathcal{H}_{\psi}
\sum_{\vec{l}_{k-1}\in\mathcal{Z}_{k-1}}
\sum_{l_k\in \mathbb{Z}\setminus \{0\} }
			y^\sigma_{\vec{l}_{k-1}}
			c^\sigma_{l_k}
			 e^{\mathi\vec{l}_{k-1}\cdot(\vec{\theta}_{k-1}+\vec{n}_{k} \psi)}
			e^{\mathi{l}_{k}\cdot(\theta_k+{n}_{k} \psi)}\\
=& - \mathi
\sum_{\vec{l}_{k-1}\in\mathcal{Z}_{k-1}}
\sum_{l_k\in \mathbb{Z}\setminus \{0\} }
			\tmop{sign}(\vec{l}_{k-1}\cdot\vec{n}_{k-1}+l_k n_k)\ 
			y^\sigma_{\vec{l}_{k-1}}
			c^\sigma_{l_k}\\
            &\qquad \qquad
			 \cdot e^{\mathi \vec{l}_{k-1}\cdot(\vec{\theta_{k-1}}+\vec{n}_{k-1}\psi)}
			e^{\mathi l_k (\theta_k+n_k \psi)} \backassign    \mathcal{H}_{\psi}(\star).     
\end{align*}

Choosing the increasing sequence $\vec{n}$ such that
$ 
n_0 = 1,\quad  n_k =  2 N_{k-1} n_{k-1},\ \forall k\geqslant 1,
$ 
and recalling that $l_k\neq 0$ allows us to estimate for all $k\geqslant 1$
\[ |\vec{l}_{k-1}\cdot\vec{n}_{k-1}|\leqslant \|\vec{l}_{k-1}\|_1\|\vec{n}_{k-1}\|_{\infty}\leqslant N_{k-1} n_{k-1}\leqslant \frac12 n_k< |l_k n_k|.
\] 
Hence $\tmop{sign}(\vec{l}_{k-1}\cdot\vec{n}_{k-1}+l_k n_k)=\tmop{sign}(l_k n_k)$ yielding
\begin{align*}
\MoveEqLeft
\mathcal{H}_{\psi}(\star)
=
- \mathi
\sum_{\vec{l}_{k-1}\in\mathcal{Z}_{k-1}}
\sum_{l_k\in \mathbb{Z}\setminus \{0\} }
			\tmop{sign}(l_k n_k)\ 
			y^\sigma_{\vec{l}_{k-1}}
			c^\sigma_{l_k}
			 \cdot e^{\mathi \vec{l}_{k-1}\cdot(\vec{\theta}_{k-1}+\vec{n}_{k-1}\psi)}
			e^{\mathi l_k (\theta_k+n_k \psi)}\\
&=   
- \mathi
\sum_{\vec{l}_{k-1}\in\mathcal{Z}_{k-1}}
y^\sigma_{\vec{l}_{k-1}} e^{\mathi \vec{l}_{k-1}\cdot(\vec{\theta}_{k-1}+\vec{n}_{k-1}\psi)}
\sum_{l_k\in \mathbb{Z}\setminus \{0\} }
\tmop{sign}(l_k n_k) c^\sigma_{l_k}e^{\mathi l_k (\theta_k+n_k \psi)}\\
&=
\mathd F^{\sigma}_k  (\vec{\theta}_{k-1} + \vec{n}_{k-1} \psi)
(\mathcal{H}\varphi^{\sigma}) (\theta_{k} + n_{k} \psi).
\end{align*}
\end{proof}
Coming back to the proof of Lemma \ref{lemma-inside-H},
for ease of notation, we write
\[
\Phi_{\vec\theta, \vec N} (\psi) = F (\vec{\theta} + \vec{n} \psi),
\quad \Phi^{\mathcal{H}}_{\vec\theta, \vec N} (\psi) = F^{\mathcal{H}} (\vec{\theta} + \vec{n} \psi),
\quad \Gamma_{\vec\theta, \vec N} (\psi) = G (\vec{\theta} + \vec{n} \psi).
\]
For a function  $F = F(\psi)$ with $\psi\in\mathbb{T}=[0,2\pi)$, we note
$ 
\mathbb{E}^{\psi} F \assign  \int_0^{2\pi} F(\psi) \frac{\mathd \psi}{2\pi} .
$ 
\begin{remark}\label{theta-and-psi}
  There hold, if $\mathcal{T}$ denotes either $\mathcal{H}$ or the identity (no entry):
  \begin{align*}
  \MoveEqLeft
    \big| \mathbb{E}^{\vec\theta} \big\langle F^{\mathcal{T}} (\vec\theta), G (\vec\theta)
    \big\rangle_{X, X^{\ast}} \big| 
     =  \big| \mathbb{E}^{\psi} \mathbb{E}^{\vec\theta}
    \big\langle F^{\mathcal{T}} (\vec\theta), G (\vec\theta) \big\rangle_{X, X^{\ast}} \big|\\
    & =  \big| \mathbb{E}^{\psi} \mathbb{E}^{\vec\theta} \big\langle \Phi_{\vec\theta,
    \vec N}^{\mathcal{T}} (\psi), \Gamma_{\vec\theta,\vec N} (\psi) \big\rangle_{X, X^{\ast}} \big|.
  \end{align*}  
\end{remark}

    The first equality is obvious as the integrand does not depend upon $\psi$. 
  We explain the second equality above. It follows immediately if we observe that $|\mathbb{E}^{\theta}  
  \langle \phi^{\mathcal{T}}_{\vec\theta,\vec N} (\psi), \Gamma_{\vec\theta,\vec N} (\psi)
  \rangle_{X, X^{\ast}} |$ does not depend upon $\vec N$ or $\psi$. Indeed, the
  terms that arise are of the form
  \[ \big\langle \mathd F_k^{\sigma} (\vec{\theta}_k + \vec{n}_k \psi)\mathcal{T} \varphi^{\sigma} (\theta_{k + 1} + n_{k + 1} \psi),
     \mathd G_l^{\eta} (\vec{\theta}_l + \vec{n}_l \psi)
     \varphi^{\eta} (\theta_{l + 1} + n_{l + 1} \psi) \big\rangle_{X, X^{\ast}} .
  \]
  A successive integration in the
  $\theta_k$'s shows by periodicity of the involved functions an
  independence from $\psi$ and $\vec N$.

  We now use these observations together with Remark \ref{theta-and-psi}, Lemma \ref{moveH}, and Remark \ref{x-and-theta} to continue our estimate
  \begin{align*}
  \MoveEqLeft
    \big|\mathbb{E}^{\vec\theta} \big\langle F^{\mathcal{H}} (\vec\theta), G
  (\vec\theta) \big\rangle_{X, X^{\ast}} \big|
     = 
    \big| \mathbb{E}^{\psi} \mathbb{E}^{\vec\theta} \big\langle \Phi_{\vec\theta,\vec
    N}^{\mathcal{H}} (\psi), \Gamma_{\vec\theta,\vec N} (\psi) \big\rangle_{X, X^{\ast}} \big|
    \\
     =& 
    \big| \mathbb{E}^{\psi} \mathbb{E}^{\vec\theta} \big\langle \mathcal{H}_{\psi}
    \Phi_{\vec\theta,\vec N} (\psi), \Gamma_{\vec\theta,\vec N} (\psi) \big\rangle_{X, X^{\ast}}\big|
    \\
     \leqslant & h_p \big(\mathbb{E}^{\psi} \mathbb{E}^{\vec\theta}|\Phi_{\vec\theta,\vec N} (\psi)|_X^p\big)^{1/p}
    \ \big(\mathbb{E}^{\psi} \mathbb{E}^{\vec\theta}|\Gamma_{\vec\theta,\vec N} (\psi)|_{X^{\ast}}^q\big)^{1/q}
    \\
	 = & 
    h_p
    \big(\mathbb{E}^{\vec\theta} |F (\vec\theta)|_X^p\big)^{1/p}
    \ \big(\mathbb{E}^{\vec\theta} |G (\vec\theta)|_{X^{\ast}}^q\big)^{1/q}
     =  h_p\ \| F \|_{L_X^p} \| G \|_{L_{X^{\ast}}^q}.
  \end{align*}
This concludes the proof of Lemma \ref{lemma-inside-H}.
\end{proof*}

\paragraph{The final estimate in the weak form}

For $k, l \geqslant 0$ we consider terms of the form
\[\mathbb{E}^{\vec{\theta}}
 \big\langle \mathd F_k^{\sigma} (\vec{\theta}_k )\mathcal{H} \varphi^{\sigma} (\theta_{k + 1} ),
     \mathd G_l^{\eta} (\vec{\theta}_l )
     \varphi^{\eta} (\theta_{l + 1} ) \big\rangle_{X, X^{\ast}} 
\]
and 
\[\mathbb{E}^{\vec{\theta}}
\big\langle  \mathd F^{\sigma}_k(\vec{\theta}_k)
  \Sha\varphi^{\sigma} (\theta_{k + 1}), \mathd G_l^{\eta} (\vec\theta_l) \varphi^{\eta}
  (\theta_{l + 1}) \big\rangle_{X, X^{\ast}}.
\]

Concerning the dualized form involving the Hilbert transform we first note that if $l\neq k$ begin by integrating in $\theta_{s+1}$ with $s=\max\{k,l\}$ and observe that the contribution is 0 as $\mathcal{H}\varphi^\sigma$ and $\varphi^\sigma$ have mean 0. If $l=k$ begin by integrating in $\theta_{k+1}$ and observe that 
  $\mathcal{H}\varphi^\sigma \cdot \varphi^{\sigma}$ still has mean zero and these cases yield no contribution.
  The only arising terms are thus those with $k=l$ and $\bar\sigma=\eta$. 

Concerning the dualized form involving $\Sha$ we again note that there is no contribution if $l\neq k$ by integrating first with respect to $\theta_{s+1}$. Since $\Sha \varphi^\sigma=\sigma\Sha\varphi^{\bar\sigma}$ , we see that again the only arising terms are those with $k=l$ and $\bar\sigma=\eta$.
Thus by Lemma \ref{lemma-projection-hilbert} and Remark \ref{flattening-hilbert}
\begin{align*}
  \MoveEqLeft \mathbb{E}^{\theta_{k + 1}}  (\mathcal{H} \varphi^{\sigma})  (\theta_{k +
  1}) \varphi^{\eta}  (\theta_{k + 1}) 
   =  \mathbb{E}^{\theta_{k + 1}}  (\mathcal{P}
  \mathcal{H} \varphi^{\sigma})  (\theta_{k + 1}) \varphi^{\eta}  (\theta_{k +
  1})\\
   =&  c_0 \mathbb{E}^{\theta_{k + 1}}  (\Sha \varphi^{\sigma}) 
  (\theta_{k + 1}) \varphi^{\eta}  (\theta_{k + 1}).
\end{align*}
This is a decisive feature of the similarity between $\Sha$ and $\mathcal{H}$.
We have thus
$  \mathbb{E}^{\theta}  \langle F^{\mathcal{H}} (\theta), G (\theta)
   \rangle_{X, X^{\ast}} = c_0 \mathbb{E}^{\theta}  \langle \Sha F
   (\theta), G (\theta) \rangle_{X, X^{\ast}}  $, when $\langle f \rangle_{I_0} = 0$ and $(f, h_{I_0}) =
0$
This implies, using Remark \ref{x-and-theta}, and assuming $\langle f \rangle_{I_0} = 0$
and $(f, h_{I_0}) = 0$, that
\begin{align*}
\MoveEqLeft \big| \mathbb{E}^x \big\langle \Sha f (x), g (x) \big\rangle_{X, X^{\ast}} \big| 
  =  \big|\mathbb{E}^{\vec\theta} \big\langle \Sha F (\vec\theta), G (\vec\theta)
  \big\rangle_{X, X^{\ast}} \big|   \\
  = & c^{- 1}_0  \big|\mathbb{E}^{\vec\theta} \big\langle F^{\mathcal{H}} (\vec\theta), G
  (\vec\theta) \big\rangle_{X, X^{\ast}} \big| 
  \leqslant  h_p c^{- 1}_0 \|F\|_{L_{X,\vec\theta}^p} \|G\|_{L_{X^{\ast},\vec\theta}^q} \\
  = & h_p c^{- 1}_0 \|f\|_{L_{X}^p} \|g\|_{L_{X^{\ast}}^q}.
\end{align*}
To finish the estimate for a general $f$ defined on $I_0$,
extend this function by zero on a large dyadic interval $J\supset I_0$.
Setting $\tilde{f} = f - \langle f \rangle_{J} \tmmathbf{1}_{J} - (f, h_{J})
h_{J} $ and recalling the proof of Theorem \ref{T: comparison Lp norms} ensures that
$\| \Sha^J f \|_{L_X^p} = \| \Sha^J \tilde{f} \|_{L_X^p} + O(1/|J|)$.
Since all operators $\Sha^J$ have the same $L^p$ norm,
this yields the desired estimate $s_p \leqslant h_p c^{- 1}_0$ in the limit $|J|\to\infty$.
This completes the proof of Theorem \ref{theorem_lower}.
%
\section{The constant $c_0$}\label{S: Constant_c0}
In order to estimate the constant $c_0$, we are using the identity
$ 
\mathcal{P} \mathcal{H} \varphi^- = - c_0 \varphi^+ = c_0
     \Sha \varphi^-.
$ 
Let $\beta(x) = \tmmathbf{1}_{[0,\pi)}$. We have $\varphi^-(x) = \beta(x) - \beta(x+\pi) = 2 \beta(x) - 1$,
thus $\mathcal{H}\varphi^- = 2 \mathcal{H}\beta$, that is
\[
\mathcal{H}\varphi^-(x) = \frac{1}{2\pi} p.v. \int_0^{\pi} \cot\big( \frac{x-t}{2} \big) \mathd t.
\]
Observe that
$ 
 \frac{\partial}{\partial t}\left( -2 \log \left| \sin \left( \frac{x-t}{2} \right) \right| \right)
	=
	\cot\left( \frac{x-t}{2} \right).
$ 
In a principal value sense, one calculates that
$ 
	\mathcal{H} \varphi^-(x) = \frac{2}{\pi} \log \tan\left( \frac{x}{2} \right)
$ 
when $x\in(0,\pi/2)$. It remains to calculate the following mean on $(0,\pi/2)$,
\[
	\frac{1}{\pi/2} \int_0^{\pi/2}
		\frac{2}{\pi} \log \tan\big( \frac{x}{2} \big) \mathd x
	=
	-8 G / \pi^2,
\]
where $G$ denotes the Catalan constant $G=\sum_{k=0}^\infty (-1)^k \frac{1}{(2k+1)^2} \sim 0.91597$.
Comparing with $\varphi^+(x) = 1$ on $(0,\pi/2)$ yields
$ 
	c_0 = 8 G / \pi^2 \sim 0.742454$ and $ c_0^{-1} \sim 1.34689.
$ 
\begin{figure}[ht]
	\centering
	\includegraphics[width=6cm]{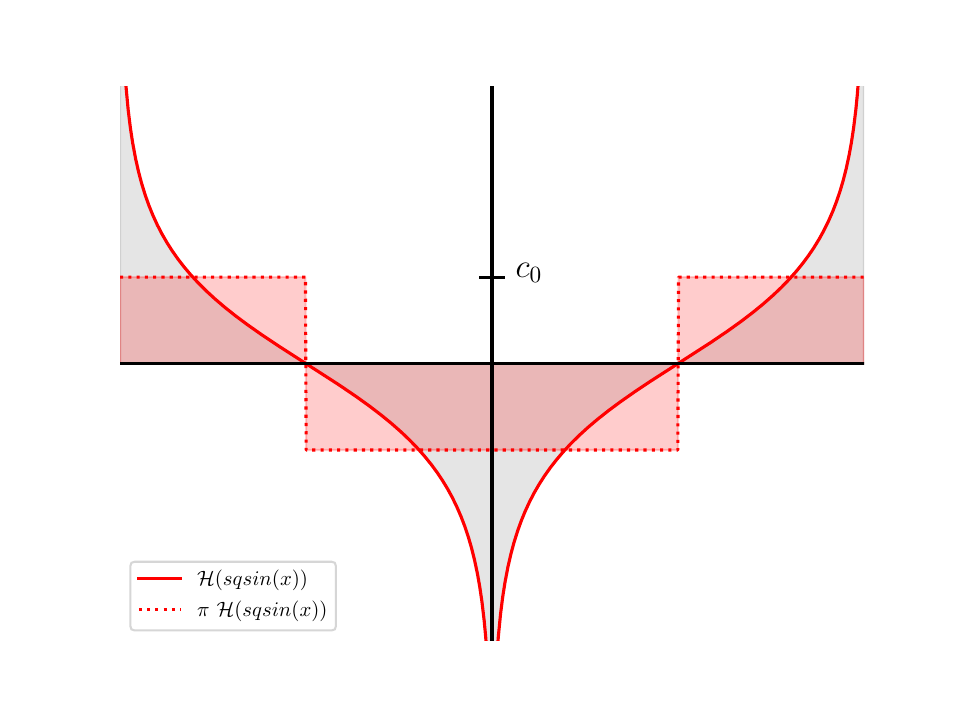}
	\caption{The constant $c_0$.}
\end{figure}
\section{Averaging of the dyadic Hilbert transform}\label{S: averaging dyadic Hilbert transform}
We prove that the average of the operators $\Sha$ in the
sense of {\cite{Pet2000}} is null. For that, let $\mathcal{D}^{\alpha, r} =
\{ 2^r I + \alpha : I \in \mathcal{D} \}$ the dilated and translated dyadic
grid on $\mathbb{R}$. Here, let $1 \leqslant r < 2$ and $\alpha \in \mathbb{R}$. Denote by
$h^{\alpha, r}_I$ the corresponding $L^2$-normalized Haar functions and by
$\mathcal{S}^{\alpha, r}_0$ the dyadic Hilbert transform associated to
$\mathcal{D}^{\alpha, r}$. Since in the usual sense,
\[ \mathcal{S}^{\alpha, r}_0 : f (x) \mapsto \sum_{I \in \mathcal{D}^{\alpha,
   r}} \big[- ( f, h^{\alpha, r}_{I_-} ) h^{\alpha, r}_{I_+} (x) +
   ( f, h^{\alpha, r}_{I_+} ) h^{\alpha, r}_{I_-} (x)\big], \]
the kernel of $\mathcal{S}^{\alpha, r}_0$ is
\[ K^{\alpha, r}_0 (t, x) = \sum_{I \in \mathcal{D}^{\alpha, r}} \big[- h^{\alpha,
   r}_{I_-} (t) h^{\alpha, r}_{I_+} (x) + h^{\alpha, r}_{I_+} (t) h^{\alpha,
   r}_{I_-} (x)\big] . \]
For each fixed $I$ the product $- h^{\alpha,
   r}_{I_-} (t) h^{\alpha, r}_{I_+} (x) + h^{\alpha, r}_{I_+} (t) h^{\alpha,
   r}_{I_-} (x)$ is supported in $I_-\times I_+ \cup I_+\times I_-   \subset I\times I$. 
   Below is the illustration of the sign distribution of the two products of Haar functions in $I\times I$. 

\begin{figure}[ht]
	\centering
	\includegraphics[width=2.93cm,height=2.43cm]{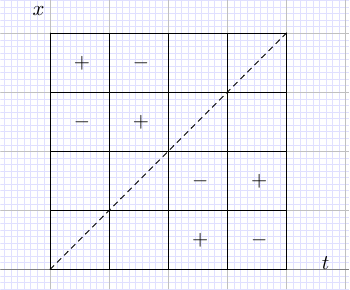}
	\caption{sign distribution}
\end{figure}

Following the strategy of the second author in {\cite{Pet2000}}, the average
of the kernel by dilation and translation we consider is:
\[ \mathbb{E}_r \mathbb{E}_{\alpha} K^{\alpha, r}_0 (t, x) = \frac{1}{\log 2}
   \int^2_1 \lim_{R \rightarrow \infty} \frac{1}{2 R} \int^R_{- R} K^{\alpha,
   r}_0 (t, x) \mathd \alpha \frac{\mathd r}{r} . \]
Via the explanations in {\cite{Pet2000}}, the resulting average depends upon $t-x$, is antisymmetric and of homogeneity $-1$, such as uniquely satisfied by the function $\frac{c}{x-t}$ for some $c\in \mathbb{R}$. This can be verified also by an explicit calculation. In {\cite{Pet2000}} it was vital that $c\neq 0$. To show this, the inner average over translations was computed, which was sufficient to conclude.

The average over translations can be visualized by moving the above square along the diagonal. 
The result of the expectation over translations can then be determined by linearly interpolating the values at distances $0$, $\frac1{4\sqrt{2}}  | I |$, $\frac1{2\sqrt{2}}  | I |$, $\frac3{4\sqrt{2}}  | I |$, $\frac1{\sqrt{2}}  | I |$ from the diagonal $x=t$ and then summing up the contributions. The average over dilations can then be calculated second, if desired. It is somewhat easier however to calculate the average over dilations first such as done in \cite{PTV}. The proof we present here however, does not require any calculation. 
\begin{lemma}
  We have $\mathbb{E}_r \mathbb{E}_{\alpha} K^{\alpha, r}_0 (t, x) = 0.$
\end{lemma}

\begin{proof}
  The kernel $K^{\alpha, r}_0 (t, x)$ can be split into a sum of
  partial kernels
  \[ K_-^{\alpha, r} (t, x) = \sum_{I \in \mathcal{D}^{\alpha, r}} h^{\alpha,
     r}_{I_+} (t) h^{\alpha, r}_{I_-} (x) ,
  \qquad
   K_+^{\alpha, r} (t, x) = \sum_{I \in \mathcal{D}^{\alpha, r}} -
     h^{\alpha, r}_{I_-} (t) h^{\alpha, r}_{I_+} (x) . \]
  $K_-^{\alpha, r} (t, x)$ is supported under the diagonal, $t > x$, while
  $K_+^{\alpha, r} (t, x)$ is supported over the diagonal, $t < x$. Clearly
  the averages produce no mass on the diagonal. Observe now that the average
  of the {\it even} kernel $K_-^{\alpha, r} (t, x) - K_+^{\alpha, r} (t, x)$ is also depending upon $x-t$, has homogeneity $-1$ but is symmetric, thus characterised by $\frac{c}{|x-t|}$. The resulting kernel operator is thus dilation invariant, translation invariant and symmetric. It therefore represents the zero operator and $c=0$. We may then conclude that the averages of the partial kernels themselves are 0 because they are confined to separate supports split by the diagonal $x=t$. The claim follows.
\end{proof}


\begin{thebibliography}{X}

 \bibitem{BK} R.~Ba{\~n}uelos, M.~Kwa\'{s}nicki.
{\em On the $\ell_p$ norm of the discrete Hilbert transform},
Duke Math. J. 168(3): pp. 471--504, 2019.

 \bibitem{BanWan1996}R.~Ba{\~n}uelos, G.~Wang.
  {\em Orthogonal martingales under differential subordination and
  applications to Riesz transforms}, I
  llinois J. Math., 40(4): pp. 678--691, 1996.


\bibitem{Bou1983a} J.~Bourgain. 
{\em Some remarks on Banach spaces in which martingale difference sequences are unconditional},
Ark. Mat. 21(2): pp. 163--168, 1983.


 \bibitem{Bur1983a}D.~L.~Burkholder. 
  {\em A geometric condition that implies the existence of certain singular integrals of Banach-space-valued functions}, 
  Conference on harmonic analysis in honor of Antoni Zygmund, Vol. I, II (Chicago, Ill., 1981), Wadsworth Math. Ser., 
    270--286. Wadsworth, Belmont, CA, 1983.
    
    
\bibitem{Bur1984a} D.~L.~Burkholder.
{\em Boundary value problems and sharp inequalities for martingale transforms},
Ann. Probab., 12(3): pp. 647--702, 1984.
    
 \bibitem{Bur2001a} D.~L.~Burkholder.
 {\em Martingales and singular integrals in Banach spaces}, 
 Handbook of the geometry of Banach spaces, Vol. I,
 North-Holland, Amsterdam, 2001.
 
 \bibitem{DomPet2014} K.~Domelevo, S.~Petermichl.
{\em Sharp {$L^p$} estimates for discrete second order {R}iesz transforms},
Adv. Math., 262: pp. 932--952, 2014.

\bibitem{DPTV} K.~Domelevo, S.~Petermichl, S.~Treil, A.~Volberg.
{\em The matrix $A_2$ conjecture fails, i.e. $3/2>1$},
arXiv:2402.06961, pp. 1--46, 2024.
 
 \bibitem{Ess1984} M.~Ess{{\'e}}n.
{\em A superharmonic proof of the {M}. {R}iesz conjugate function theorem},
Ark. Mat., 22(2): pp. 241--249, 1984.
 
 \bibitem{Fig1990} T.~Figiel.
{\em Singular integral operators: a martingale approach}, 
London Math. Soc. Lecture Note Ser., 158,
Cambridge Univ. Press, Cambridge, 1990.

  \bibitem{GeiMonSak2010} S.~Geiss, S.~Montgomery-Smith, E.~Saksman.
 {\em On singular integral and martingale transforms}, 
  Trans. Amer. Math. Soc., 362(2): pp. 553--575, 2010.
  
 \bibitem{H}  T.~Hyt\"{o}nen.
  {\em The sharp weighted bound for general Calder\'{o}n–Zygmund operators}
Ann. Math. 175(3): pp. 1473--1506, 2012.

\bibitem{HytNeeVerWei2016a}T.~Hyt\"{o}nen, J.~van~Neerven, M.~Veraar, L.~Weis.
{\em Analysis in Banach spaces}, 
Vol. I. Martingales and
  Littlewood-Paley theory, volume~63 of Ergebnisse der Mathematik und
  ihrer Grenzgebiete. 3. Folge. Springer, Cham, 2016.
  
  \bibitem{KT} S.~Kakaroumpas, S.~Treil. 
  {\em ``{S}mall step'' remodeling and counterexamples for weighted estimates with arbitrarily ``smooth'' weights}, 
  Adv. Math., 376: pp. 1--52, 2021. 
  
   \bibitem{KloPla1992} P.~E.~Kloeden, E.~Platen.
{\em Numerical solution of stochastic differential equations},  
volume~23  of Applications of Mathematics (New York). Springer-Verlag, Berlin, 1992.
  
 \bibitem{Lacey} M.~T.~Lacey. 
 {\em Two-weight inequality for the Hilbert transform: A real variable characterization, II}, 
Duke Math. J. 163(15): pp. 2821--2840, 2014.
  
  \bibitem{Law2010a}G.~F.~Lawler. 
  {\em Random walk and the heat equation},  
  volume~55  of Student Mathematical Library.
 American Mathematical Society, Providence, RI, 2010.
  
\bibitem{Nazarov} F.~Nazarov. 
{\em A counterexample to Sarason's conjecture}, unpublished manuscript, available at \url{http://users.math.msu.edu/users/fedja/prepr.html}
  
  \bibitem{NPTV} F.~Nazarov, G.~Pisier, S.~Treil, A.~Volberg. 
  {\em Sharp estimates in vector Carleson imbedding theorem and for vector paraproducts}, 
  J. reine angew. Math. 542: pp. 147--171, 2002.
  
    \bibitem{NTVnonhom} F.~Nazarov, S.~Treil, A.~Volberg. 
 {\em  The Tb-theorem on non-homogeneous spaces},
Acta Math. 190(2): pp. 151--239, 2003. 
  
 \bibitem{Pet2000}S.~Petermichl.
   {\em Dyadic shift and a logarithmic estimate for Hankel operators with matrix symbol},
   C. R. Acad. Sci. Paris S\'er. I Math., 330(6): pp. 455--460, 2000.

  \bibitem{PetPot2003a}S.~Petermichl, S.~Pott. 
  {\em A version of Burkholder's theorem for operator-weighted spaces},
  Proc. Amer. Math. Soc., 131(11): pp. 3457--3461, 2003.
  
  \bibitem{Pet2007}  S.~Petermichl. 
  {\em The sharp bound for the Hilbert transform on weighted Lebesgue spaces in terms of the classical $A_p$ characteristic}, 
  Amer. J. Math. 129, no. 5, pp. 1355--1375, 2007.
  
  \bibitem{PTV} S.~Petermichl, S.~Treil, A.~Volberg.
  {\em Why the Riesz transforms are averages of the dyadic shifts?}
Publ. Mat. 46, pp. 209--228, 2002.

\bibitem{PV2002}  S.~Petermichl, A.~Volberg. 
{\em Heating of the Ahlfors-Beurling operator: weakly quasiregular maps on the plane are quasiregular}, 
Duke Math. J. 112, no. 2: pp. 281--305, 2002.

  \bibitem{Pic1972} S.~K.~Pichorides.
{\em On the best values of the constants in the theorems of {M}. {R}iesz,
  {Z}ygmund and {K}olmogorov},
Studia Math., 44: pp. 165--179. (errata insert), 1972.
  
  \bibitem{Pis2016} G.~Pisier.  
  {\em Martingales in Banach spaces}, 
  Cambridge Studies in Advanced Mathematics, vol. 155, Cambridge University Press, Cambridge, 2016.

\bibitem{Pot2014}S.~Pott, A.~Stoica.
{\em Linear bounds for Calder\'on-Zygmund operators with even kernel on UMD spaces},
J. Funct. Anal., 266(5): pp. 3303--3319, 2014. 

 \bibitem{Tal1986n}D.~Talay. 
 {\em Discr{\'e}tisation d'une
  {\'e}quation diff{\'e}rentielle stochastique et calcul approch{\'e}
  d'esp{\'e}rances de fonctionnelles de la solution},
  RAIRO Mod{\'e}l. Math. Anal. Num{\'e}r.,
  20(1): pp. 141--179, 1986.

\bibitem{Tre2013a} S.~Treil.
 {\em Sharp $A_2$ estimates of Haar shifts via Bellman function},
{{Recent trends in analysis. Proceedings of the conference in
  honor of Nikolai Nikolski on the occasion of his 70th birthday}}, Bucharest: The
  Theta Foundation, pp. 187--208, 2013.
  
  \bibitem{TreVol1997} S.~Treil, A.~Volberg.
{\em Wavelets and the angle between past and future},
J. Funct. Anal., 143(2): pp. 269--308, 1997.


\end{thebibliography}
\end{document}